\documentclass[10pt]{article}
\usepackage{amsfonts}
\usepackage{amsmath,hhline,epsfig,latexsym}
\usepackage{amssymb}
\usepackage{graphicx}
\usepackage[english]{babel}
\usepackage{hyperref}
\usepackage{color}
\makeatletter
\renewcommand{\paragraph}{\@startsection{paragraph}{4}{0ex}%
   {-3.25ex plus -1ex minus -0.2ex}%
   {1.5ex plus 0.2ex}%
   {\normalfont\normalsize\bfseries}}
\makeatother

\stepcounter{secnumdepth}
\stepcounter{tocdepth}
\usepackage{amsmath}
\providecommand{\U}[1]{\protect\rule{.1in}{.1in}}
\newtheorem{thm}{Theorem}
\newtheorem{Assumption}{Assumption}
\newtheorem{propo}{Proposition}

\newtheorem{rmq}{Remark}

\newtheorem{lemma}{Lemma}
\newenvironment{proof}[1][Proof]
{\noindent\textbf{#1:} }{\hfill\rule{0.5em}{0.5em}}

\def\dis{\displaystyle}

\def\0vec{\mathbf{0}}

\def\om{\omega}
\def\Om{\Omega}

\def\\Phivec{\mathbf{\Phi}}
\def\gvec{\mathbf{g}}

\def\wvec{\mathbf{w}}

\newcommand{\ii}{\iiTL }
\newcommand{\iiTL}{\int_0^T\!\!\!\!\int_0^L }
\newcommand{\be}{\begin{equation}}
\newcommand{\ee}{\end{equation}}
\newcommand{\ba}{\begin{eqnarray}}
\newcommand{\ea}{\end{eqnarray}}
\title{\textbf{On the One-Dimensional Nonlinear Monodomain Equations with Moving Controls}}

\author{
	\textsc{Karl Kunisch}\thanks{Institut f\"ur Mathematik und Wissenschaftliches Rechnen, Karl-Franzens-Universit\"at, Heinrichstra\ss e 36, 8010 Graz,
Austria and Johann Radon Institute for Computational and Applied Mathematics, \"Osterreichische Akademie der Wissenschaften. E-mail: {\tt karl.kunisch@uni-graz.at}. Work partially supported
	by the ERC advanced grant 668998 (OCLOC) under the EU's H2020 research program.}\quad
	\and
	\textsc{Diego A. Souza}\thanks{Department of Mathematics, Federal University of Pernambuco, UFPE, CEP 50740-545, Recife,
PE, Brazil. E-mail: {\tt diego.souza@dmat.ufpe.br}. Work partially supported
	by the ERC advanced grant 668998 (OCLOC) under the EU's H2020 research program.}}
\date{}

\begin{document}
\maketitle



\begin{abstract}
	In this paper local exact controllability to the trajectories for the one-dimensional monodomain
	equations with the FitzHugh-Nagumo and Rogers-McCulloch ionic models using distributed
	controls with a moving support is investigated. In a first step a new Carleman inequality for the adjoint of the
	linearized monodomain equations, under assumptions on the movement of the control region, is
	presented. It leads to null controllability at any positive time. Subsequently, a local result concerning
	the exact controllability to the trajectories for the nonlinear monodomain equations is deduced.
\end{abstract}

\

\noindent {\bf keyword:} exact controllability to the trajectories, monodomain equations, moving controls,
	FitzHugh-Nagumo ionic model, Rogers-McCulloch ionic model, heat equation with memory.
	
	\
	



\section{Introduction}

	The main objective of this paper is to study controllability properties of nonlinear reaction-diffusion
	systems which model the electrical activity in the heart. Our reference model for the heart's electrical
	activity is the so-called {\it bidomain model}, formulated mathematically in \cite{tung}, see also e.g.
	\cite[Chapter $2$]{sundnes} and the references therein. Next, we describe  such a model in order to
	motivate the controllability results studied. Since the  bidomain model is not frequently discussed in
	mathematical publications we allow more space for our description. At the end of this section we shall
	relate our results to results on the exact controllability of the heat equation with memory terms.

	We start by saying that heart muscle cells belong to a class of cells known as {\it excitable cells}, which
	have the ability to respond actively to an electrical stimulus. In absence of an electrical stimulus cells
	remain electrically quiescent at a given potential difference across the cell membrane. At rest the potential
	inside the cells, called the {\it intracellular potential}, is negative compared to the {\it extracellular potential},
	which is the potential in the interstitial space between the cells and the potential difference is referred to as
	{\it transmembrane voltage}. When such cells are stimulated electrically they depolarize the transmembrane
	voltage towards less negative or positive values. If the delivered stimulus is strong enough to depolarize the
	cell above an intrinsic firing threshold an active response is elicited, otherwise the cell returns to its resting
	state. The active response of the cell is of \emph{all or none} type, that is, the elicited active response is always
	the same independently of the applied stimulus strength. This threshold behavior discriminating between
	active non-linear and passive linear response is referred to as \emph{excitability}. The depolarization of cells
	above the firing threshold is a very fast process which is followed by a slower {\it repolarization} that restores
	the potential difference to its resting value. The complete cycle of depolarization and repolarization is
	called an {\it action potential}. In tissue the intracellular spaces of cells are interconnected and thus an
	ongoing action potential in one cell can depolarize the resting potential in neighboring cells up to the firing
	threshold and thus provide a mechanism for the propagation of electrical signals. This ability enables an
	electric activation occurring in one part of the heart to propagate through the muscle and activate the
	entire heart.

	The bidomain model is a macroscopic model based on the assumption that, at mesoscopic scale, cardiac
	tissue can be viewed as partitioned into two ohmic conducting media separated by the cell membrane: the
	{\it intracellular} medium formed by the interior space of cardiac cells and the {\it extracellular} medium
	which represents the space between cells. Both domains are assumed to be continuous, and they both fill
	the complete volume of the heart muscle. This latter assumption of \emph{interpenetrating domains}, that
	is intracellular space, extracellular space and membrane co-exist at any point in space, does not reflect
	biophysical reality at a cellular size scale, but can be justified at a mesoscopic size scales based on homogenization
	arguments. The justification for viewing the intracellular space as continuous is that the muscle
	cells are interconnected via conducting pores referred to as {\it gap junctions}. Because of the gap junctions,
	substances such as ions or small molecules may pass directly from one cell to another, without entering the space
	between the two cells (the extracellular domain). Having said this, in each of the two domains a macroscopic
	electric potential is defined and the membrane acts as an electrical insulator between the two domains,
	since otherwise we could not have a potential difference between the intracellular and extracellular domains.
	Although the resistance of the cell membrane itself is very high, it allows electrically charged molecules (ions)
	to pass through specific channels embedded in the membrane. Then an electrical current referred to as
	\emph{ionic current} will cross the membrane, the magnitude of which will depend on the \emph{driving
	force} across the membrane, that is the difference between transmembrane voltage and the equilibrium potential
	for a given ion species, and on the channel's permeability to this ion species. The {\it transmembrane
	voltage} is defined as the potential difference across the membrane for every point in the heart. The bidomain
	formulation recognizes that cardiac tissue is electrically anisotropic and that current flows in both extracellular
	and intracellular domains. Bidomain models are necessary to simulate defibrillation and the biophysical
	mechanisms underlying the initiation of propagation with pacing stimuli where the current is injected in the
	extracellular domain, and have the further advantage that the most commonly measured cardiac electrical signals
	-- extracellular potential and transmembrane voltage -- are direct model outputs.

	In the following let $\Omega$ denote a sample cardiac tissue in dimension two or three and denote
	by $\nu$ the outward unit normal vector. Let $u_i = u_i(t, x)$ and $u_e = u_e(t, x)$ be the intracellular
	and extracellular electric potentials, respectively, and  denote by $v = u_i - u_e$ the transmembrane
	electric potential. The anisotropic properties of the media are modeled by intracellular and extracellular
	conductivity tensors $\sigma_i(x)$ and $\sigma_e(x)$, respectively. We assume that
	$\sigma_e,\,\sigma_i\in C^1(\overline\Om;\mathcal{M}_{N}(\mathbb{R})$
	($N\leq3$) with $(\sigma_e(x)\xi,\xi)\geq \sigma_{e,0}|\xi|^2$
    in~$\overline\Om~(\sigma_{e,0}>0)$ and $(\sigma_i(x)\xi,\xi)\geq \sigma_{i,0}|\xi|^2$
    in~$\overline\Om~(\sigma_{i,0}>0)$. In this way the propagation of the
    electrical signal through the cardiac tissue is described by the following parabolic system:
\begin{equation}\label{eq:bidomain}
    \left\{
        \begin{array}{lcl}
        a_m(c_mv_t+\mathcal{I}_{ion})-\nabla\cdot[\sigma_i(x)\nabla u_i]=\mathcal{I}_{s,i}
        &\mbox{in}&    \Omega \times \mathbb{R}_{>0},         \\
        \noalign{\smallskip}\dis
        a_m(c_mv_t+\mathcal{I}_{ion})+\nabla\cdot[\sigma_e(x)\nabla u_e]=-\mathcal{I}_{s,e}
        &\mbox{in}&    \Omega \times \mathbb{R}_{>0},         \\
        \noalign{\smallskip}\dis
        [\sigma_i(x)\nabla u_i]\cdot \nu = 0
        &\mbox{on}&    \partial \Omega \times \mathbb{R}_{>0},    \\
        \noalign{\smallskip}\dis
        [\sigma_e(x)\nabla u_e]\cdot \nu = 0
        &\mbox{on}&    \partial \Omega \times \mathbb{R}_{>0},    \\
        \noalign{\smallskip}\dis
        v(\cdot,0) =v_0
        &\mbox{in}&    \Omega,
        \end{array}
    \right.
\end{equation}
    where $c_m>0$ is the {\it capacitance of the cell membrane}, $a_m>0$ is the homogenized {\it surface-to-volume ratio of the cell membrane},
    $\mathcal{I}_{ion}$ is the {\it ionic current across the membrane}, $\mathcal{I}_{s,i}$ and $\mathcal{I}_{s,e}$ model the {\it intracellular} and
    {\it extracellular stimulation currents} used to trigger the action potential of the cell.
    The boundary conditions imply that the heart is surrounded by a non-conductive medium, and thus we require
    that the normal component of both the intracellular and extracellular current to be zero.
   The signs $+$ and $-$ reflect the change in the current density in extracellular and intracellular regions with the current flow across the membrane.

    The ionic current term $\mathcal{I}_{ion}$ in \eqref{eq:bidomain} for a given ion species is a function of the
    transmembrane voltage, the equilibrium potential of the ion species and additional cellular state variables $\wvec$  (ionic concentrations and gating variables).
    Let us explain briefly the appearance of these state variables: although the cell membrane itself is impermeable to ions, it has embedded in it a number of
    large proteins that form channels through the membrane where the ions can pass. Some transport proteins form pumps and
    exchangers, which are important for maintaining the correct ionic concentrations in the cells. Both pumps and exchangers
    have the ability to transport ions in the opposite direction of the flow generated by concentration gradients and electrical
    fields.
    In addition to the pumps and exchangers, certain proteins form channels in the membrane,
    through which ions may flow passively along the direction of the electrochemical gradient
    which is a function of transmembrane voltage and ion concentrations.
    These channels are extremely important for the behavior of excitable cells
    because most of the channels are highly selective regarding which ions are allowed to pass.
    This property of the channels is essential for generating and maintaining the potential difference across the membrane.
    The channels, so-called {\it gating channels}, also have the ability to open and close
    in response to changes in the transmembrane voltage or the presence of ligand molecules,
    and this ability is essential for the signal propagation in excitable tissue.
    Together with the constitutive equations for the cellular state variables $\wvec$ system
    \eqref{eq:bidomain} can be rewritten as:
\begin{equation}\label{eq:bidomain_ion}
    \left\{
        \begin{array}{lcl}
a_m(c_mv_t+\mathcal{I}_{ion}(v,\wvec))-\nabla\cdot(\sigma_i(x)\nabla u_i)=\mathcal{I}_{s,i}
        &\mbox{in}&    \Omega \times \mathbb{R}_{>0},         \\
        \noalign{\smallskip}\dis
a_m(c_mv_t+\mathcal{I}_{ion}(v,\wvec))+\nabla\cdot(\sigma_e(x)\nabla u_e)=-\mathcal{I}_{s,e}
        &\mbox{in}&    \Omega \times \mathbb{R}_{>0},         \\
        \noalign{\smallskip}\dis
        \wvec_t   +\gvec(v,\wvec)= \0vec
        &\mbox{in}&    \Omega \times \mathbb{R}_{>0},         \\
        \noalign{\smallskip}\dis
        (\sigma_i(x)\nabla u_i)\cdot \nu = 0
        &\mbox{on}&    \partial \Omega \times \mathbb{R}_{>0},    \\
        \noalign{\smallskip}\dis
        (\sigma_e(x)\nabla u_e)\cdot \nu = 0
        &\mbox{on}&    \partial \Omega \times \mathbb{R}_{>0},    \\
        \noalign{\smallskip}\dis
        v(\cdot,0) =v_0,\quad \wvec(\cdot,0) =\wvec_0
        &\mbox{in}&    \Omega,
        \end{array}
    \right.
\end{equation}
    where  $\gvec(v,\wvec)$ is a vector function that depends on the
    electrophysiological behavior of the heart cells.
    For simplicity, let us consider that we have only one celullar state variable $w$ and therefore a scalar function $g$.

    Some typical models for the ionic current include the FitzHugh-Nagumo model (see \cite{FN})
\begin{equation}\label{eq:FN}
    \mathcal{I}_{ion}(v,w)=-[bv(v-a)(1-v)- c\,w]\quad \hbox{and}\quad g(v,w)=-\gamma(v-\beta w)
\end{equation}
    as well as the Rogers-McCulloch model (see \cite{RM})
\begin{equation}\label{eq:RM}
    \mathcal{I}_{ion}(v,w)=-[bv(v-a)(1-v)- c\,vw] \quad \hbox{and}\quad g(v,w)=-\gamma(v-\beta w)
\end{equation}
    where $a,\,b,\,c,\,\gamma,\,\beta$ are positive ``membrane" parameters that define the shape of the action
    potential pulse.
 	 For additional discussion of related physiological models leading to systems comparable to \eqref{eq:bidomain_ion}
	we refer to \cite{keener} and for one dimensional models \cite{Peskoff1979,plonsey}.

    Since the bidomain model for the electrical activity in the heart which is difficult to solve and analyze,
    by making an assumption on the conductivity tensors $\sigma_i$ and $\sigma_e$, it is possible to simplify
    the model. Precisely, if we assume {\it equal anisotropy rates}, i.e. $\sigma_e = \mu \sigma_i$,
    where $\mu$ is a constant scalar, then $\sigma_i$ can be ``eliminated" from \eqref{eq:bidomain_ion},
    resulting in
\begin{equation}\label{eq:monodomain}
    \left\{
        \begin{array}{lcl}
        a_m(c_mv_t+\mathcal{I}_{ion}(v,w))
        -{1\over1+\mu}\nabla\cdot(\sigma_e(x)\nabla v)
        ={1\over 1+\mu}(\mu\,\mathcal{I}_{s,i}- \mathcal{I}_{s,e})
        &\mbox{in}&    \Omega \times \mathbb{R}_{>0},         \\
        \noalign{\smallskip}\dis
        w_t   +g(v,w)= 0
        &\mbox{in}&    \Omega \times \mathbb{R}_{>0},         \\
        \noalign{\smallskip}\dis
        (\sigma_e(x)\nabla v)\cdot \nu = 0
        &\mbox{on}&    \partial \Omega \times \mathbb{R}_{>0},    \\
        \noalign{\smallskip}\dis
        v(\cdot,0) =v_0,\quad w(\cdot,0) =w_0
        &\mbox{in}&    \Omega,
        \end{array}
    \right.
\end{equation}
    this particular reduction of the bidomain model is the so-called {\it monodomain model}.

    In this paper we shall study controllability properties for a simplified 1D nonlinear monodomain version
    of \eqref{eq:monodomain}. From now on, let us consider $\mathcal{I}_{ion}$ and $g$ to be given by \eqref{eq:RM}
    (in fact, the results presented in this paper holds for both FitzHugh-Nagumo and Rogers-McCulloch models).
    Assuming that $L>0$ be a positive length and $T>0$ a positive time,
    system \eqref{eq:monodomain} takes the form:
    \begin{equation}\label{eq:controlmonodomain}
    \left\{
        \begin{array}{lcl}
        a_m(c_mv_t+\mathcal{I}_{ion}(v,w))
        -{1\over1+\mu}(\sigma_e(x) v_x)_x
        ={\mu\over 1+\mu}\,\mathcal{I}_{s,i} -{1\over 1+\mu}\mathcal{I}_{s,e}
        &\mbox{in}&(0,L) \times (0,T),    \\
        \noalign{\smallskip}\dis
        w_t   +g(v,w)= 0
        &\mbox{in}&(0,L) \times (0,T),    \\
        \noalign{\smallskip}\dis
        \sigma_e(x)v_x\big|_{x=0}=\sigma_e(x)v_x\big|_{x=L} = 0
        &\mbox{in}&(0,T),            \\
        \noalign{\smallskip}\dis
        v(\cdot,0) =v_0,\quad w(\cdot,0) =w_0
        &\mbox{in}&(0,L).
        \end{array}
    \right.
\end{equation}

    The aim is to prove that we can steer the {\it transmembrane voltage-state variable} pair
 to a desired state (the final datum of a given trajectory),
    with the help of an extracellular stimulation $\mathcal{I}_{s,e}$, acting as control in moving subset $\om$ of $(0,L)$.
      This idea of a moving control domain to guarantee controllability has been used for many different problems in the past few years. See \cite{martin_rosier_moving} for
    the pioneer work. Here we use the approach introduced in \cite{CSZR}, relying on Carleman inequalities, which allows
    us to treat problems posed on bounded domains.
    The result will give the exact controllability to the trajectories for the monodomain model \eqref{eq:controlmonodomain} as long as
    the control domain $\omega$ moves in an appropriate manner and covers the whole
    domain $(0,L)$ (in Theorem \ref{teo1} we specify the kind of control domains
    and the regularity of the target trajectory). In practice a moving controller may be unrealistic. Rather one can think of an array of localized controllers which are activated consecutively and thus span a large part of the domain. Without loss of generality we can suppose that $c_m=\mu=1$ and $a_m={1\over2}$.
    We will deal with the controllability for the following  system
\begin{equation}\label{eq:monoT}
    \left\{
        \begin{array}{lcl}
        v_t+\mathcal{I}_{ion}(v,w)
        -(\sigma_i(x) v_x)_x=\mathcal{I}_{s,i}-\mathcal{I}_{s,e} &\mbox{in}&(0,L) \times (0,T),    \\
        \noalign{\smallskip}\dis
        w_t   +g(v,w)= 0 &\mbox{in}&(0,L) \times (0,T),    \\
        \noalign{\smallskip}\dis
        \sigma_i(x)v_x\big|_{x=0}=\sigma_i(x)v_x\big|_{x=L} = 0 &\mbox{on}&(0,T),            \\
        \noalign{\smallskip}\dis
        v(\cdot,0) =v_0,\quad w(\cdot,0) =w_0 &\mbox{in}&(0,L).
        \end{array}
    \right.
\end{equation}

	We recall that for every $(v_0,w_0)\in L^2(0,L)\times L^2(0,L)$ and every
	$\mathcal{I}_{s,i},\mathcal{I}_{s,e}\in L^2((0,L)\times(0,T))$,
	there exists a unique variational solution $(v,w)$ to~\eqref{eq:monodomain} 
	that satisfies (among other things)
\[
		v\in L^2(0,T;H^1(0,L))\cap C^0([0,T];L^2(0,L))\quad \hbox{and}\quad w\in H^1(0,T;L^2(0,L)),
\]
	see for instance \cite[Theorem $30$]{BOURGAULT}. In particular the Neumann boundary condition is satisfied in the  variational sense.


  	Let us now fix a {\it trajectory} $(\bar{v},\bar{w})$, that is,
    a sufficiently regular solution to the related uncontrolled system\,
\begin{equation}\label{eq:monoTraj}
    \left\{
        \begin{array}{lcl}
        \bar v_t+\mathcal{I}_{ion}(\bar v,\bar w)
        -(\sigma_i(x)\bar  v_x)_x
        =\mathcal{I}_{s,i}            &\mbox{in}&(0,L) \times (0,T),    \\
        \noalign{\smallskip}\dis
        \bar w_t   +g(\bar v,\bar w)= 0         &\mbox{in}&(0,L) \times (0,T),    \\
        \noalign{\smallskip}\dis
        \sigma_i(x)\bar v_x\big|_{x=0}=\sigma_i(x)\bar v_x\big|_{x=L} = 0    &\mbox{on}&(0,T),            \\
        \noalign{\smallskip}\dis
        \bar v(\cdot,0) =\bar v_0,\quad \bar w(\cdot,0) =\bar w_0           &\mbox{in}&(0,L),
        \end{array}
    \right.
\end{equation}
    with $(\bar v_0,\bar w_0)\in L^2(0,L)\times L^2(0,L)$.

	The control domain $\om:[0,T]\rightarrow 2^{(0,L)}$ is required  to contain  a subset $\om_0:[0,T]\rightarrow 2^{(0,L)}$ which satisfies the following
	geometric requirements. Here $2^{(0,L)}$ stands for the set of all subsets of $(0, L)$.


\begin{Assumption}\label{AssumptionMC}
      There exist  two times $t_1,t_2$ with $0< t_1 < t_2 < T $ such that:
\begin{enumerate}
    \item[a)]  $\om_0(t)\neq(0,L)$, for all $t\in(0,T)$;
    \item[b)]  $\bigcup\limits_{t\in(0,T)}\om_0(t)=(0,L)$;
    \item[c)] $(0,L)\setminus\om_0(t)$ is nonempty and connected in $(0,L)$ for any $t\in (0,t_1]\cup [t_2,T)$;
    \item[d)]$(0,L)\setminus \om_0(t)$ has two nonempty connected components in $(0,L)$ for any  $t\in (t_1,t_2)$;
\end{enumerate}
\end{Assumption}

	Recall the definitions of some usual spaces in the context of parabolic equations with boundary conditions of Neumann type $$
	H^2_\nu(0,L):=\{u\in H^2(0,L): {\sigma_i(x)}u_x\big|_{x=0}={\sigma_i(x)}u_x\big|_{x=L}=0\}~~\hbox{and}~~ H^3_\nu(0,L):=H^3(0,L)\cap H^2_\nu(0,L).
$$

	The main result of this paper is the following:
\begin{thm}\label{teo1}
    Assume that $\sigma_i\in C^1(\overline\Om;\mathbb{R})$, $T>0$, $\mathcal{I}_{s,i}\in H^1(0,T;H^1(0,L))$, $(\bar v,\bar w)$ satisfies \eqref{eq:monoTraj} with
    $(\bar v_0,\bar w_0)\in H^{3}_\nu(0,L)\times H^{2}_\nu(0,L)$
    and that the control domain $\om:[0,T]\rightarrow 2^{(0,L)}$ contains a subset 
    satisfying Assumption \ref{AssumptionMC} above.
    Then there exists $\delta>0$ such that whenever $(v_0,w_0)\in { L^2(0,L)}\times H^{2}_\nu(0,L)$ and
\[
    \|(v_0, w_0)-(\bar v_0,\bar w_0)\|_{L^{2}(0,L)\times H^{2}_\nu(0,L)}\leq\delta,
\]
    we can find a control $\mathcal{I}_{s,e}\in L^2((0,L)\times(0,T))$,
    with $\textnormal{supp}\,\mathcal{I}_{s,e}(\cdot,t)\subset\om(t)\,~\forall t\in(0,T)$, such that the associated state $(v,w)$, solution of system \eqref{eq:monoT}, satisfies
   \begin{equation}\label{exact_condition}
     (v,w)(\cdot,T) =(\bar v,\bar w)(\cdot, T) \quad\textit{in}\quad(0,L).
\end{equation}
\end{thm}

        The property which is established in Theorem \ref{teo1} is called {\it local exact controllability to trajectories}.
        Its proof will be given at the end of Section \ref{Section:monodomain}. It relies, in part,  on arguments and results from  \cite{CSZR} and \cite{CSSOUZA} which need to be changed to
    account for the PDE-ODE coupling and to  Neumann boundary conditions.  Thus let us set $v=\bar v +y$, $w=\bar{w}+z$
    and let us use these identities in \eqref{eq:monoT}. Taking into account that
    $(\bar{v},\bar{w})$ solves \eqref{eq:monoTraj}, we find:
\begin{equation}\label{eq:monoNull}
    \!\!\!\left\{
        \begin{array}{lcl}
        y_t-(\sigma_i(x) y_x)_x+\mathcal{I}_{ion}(y,z)+3b\bar v y^2+\ell_y(\bar v,\bar w)y+ c\bar v z
        =-\mathcal{I}_{s,e}           &\mbox{in}&(0,L) \times (0,T),    \\
        \noalign{\smallskip}\dis
        z_t   +g(y,z)= 0 &\mbox{in}&(0,L) \times (0,T),    \\
        \noalign{\smallskip}\dis
        \sigma_i(x)y_x\big|_{x=0}=\sigma_i(x)y_x\big|_{x=L} = 0 &\mbox{on}&(0,T),            \\
        \noalign{\smallskip}\dis
        y(\cdot,0) =y_0,\quad z(\cdot,0) =z_0 &\mbox{in}&(0,L),
        \end{array}
    \right.
\end{equation}
    where $\ell_y(\bar v,\bar w):=3b\bar v^2-2b(1+a)\bar v+c\bar w$
    is the coefficient of the linear term in $y$ and $y_0:=v_0-\bar v_0$ and
    $z_0:=w_0-\bar w_0$ are the initial data.

    In this way the local exact controllability to the trajectories for system \eqref{eq:monoT} is reduced to a local null
    controllability problem for the solution $(y,z)$ to the nonlinear problem \eqref{eq:monoNull}.

    Notice that if we introduce the variables $p=\gamma e^{\gamma\beta t}y$ and $q= e^{\gamma\beta t}z$,
    then null controllability for system \eqref{eq:monoNull} is equivalent to null controllability for the  system
\begin{equation}\label{eq:monoNull1}
    \!\!\!\!\!\left\{
        \begin{array}{lcl}
        p_t-(\sigma_i(x) p_x)_x +\ell_p(\bar v,\bar w)p+ \ell_q(\bar v,\bar w) q+\mathcal{N}(p,q)
         = h1_{\omega}         &  \mbox{in}&    (0,L) \times (0,T),    \\
            \noalign{\smallskip}\dis
            q_t  =   p                    &  \mbox{in}& (0,L) \times (0,T),        \\
            \noalign{\smallskip}\dis
            \sigma_i(x)p_x\big|_{x=0}=\sigma_i(x)p_x\big|_{x=L} = 0            & \mbox{on}& (0,T),    \\
            \noalign{\smallskip}\dis
            p(\cdot,0) =p_0,\quad q(\cdot,0) =q_0         & \mbox{in}&   (0,L),
        \end{array}
    \right.
\end{equation}
    where $\mathcal{N}(p,q):=b\gamma^{-1}e^{-\gamma\beta t}[3\bar v-(1+a)] p^2+ce^{-\gamma\beta t}pq
    + b\gamma^{-2}e^{-2\gamma\beta t}\,p^3$ is the nonlinear term, and
    $\ell_p(\bar v,\bar w):=\ell_y(\bar v,\bar w)-\gamma\beta+ab$ and $\ell_q(\bar v,\bar w):=\gamma c\bar v$
    are the coefficients of the linear terms in $p$ and $q$, respectively, and $p_0:=\gamma y_0$ and $q_0:=z_0$ are the initial data,
    and $ h:=-\gamma e^{\gamma\beta t}\mathcal{I}_{s,e}$ is the control and $1_\om$ is the characteristic function of $\om$.

    In order to solve the latter, following a standard approach, we will first deduce (global) null controllability of
    a suitable linearized version, namely:
\begin{equation}\label{eq:monoNulllinear}
    \!\!\!\!\!\left\{
        \begin{array}{lcl}
        p_t-(\sigma_i(x) p_x)_x +\ell_p(\bar v,\bar w)p+ \ell_q(\bar v,\bar w) q
         = G+h1_{\omega}         &  \mbox{in}&    (0,L) \times (0,T),      \\
            \noalign{\smallskip}\dis
            q_t  =   p                    &  \mbox{in}& (0,L) \times (0,T),        \\
            \noalign{\smallskip}\dis
            \sigma_i(x)p_x\big|_{x=0}=\sigma_i(x)p_x\big|_{x=L} = 0            & \mbox{on}& (0,T),    \\
            \noalign{\smallskip}\dis
            p(\cdot,0) =p_0,\quad q(\cdot,0) =q_0         & \mbox{in}&   (0,L),
        \end{array}
    \right.
\end{equation}
    where $G$ belongs to a space of functions that decay exponentially as $t\rightarrow T^-$.

    Then, appropriate and rather classical arguments will be used to deduce the local null controllability of the nonlinear
    system \eqref{eq:monoNull1}.
     \begin{rmq}\label{rem1.1}
	In the case of the FitzHugh-Nagumo model, i.e. with $\mathcal{I}_{ion}$ chosen as in \eqref{eq:FN}, the minor difference
consists in $\ell_q =\gamma c$ , $\ell_y(\bar v,\bar w):=3b\bar
v^2-2b(1+a)\bar v+c$, and $\mathcal{N}(p,q):=b\gamma^{-1}e^{-\gamma\beta
t}[3\bar v-(1+a)] p^2+ce^{-\gamma\beta t}q
     + b\gamma^{-2}e^{-2\gamma\beta t}\,p^3$.\newline
    \noindent
    Let us note in particular that for $q_0=0$, $\ell_q=1$ and $G=0$ and setting $d(x,t)=\ell_p(\bar v, \bar w)$
system \eqref{eq:monoNulllinear} becomes
    \begin{equation}\label{eq:heatmemory}
     \left\{
         \begin{array}{lcl}
          \noalign{\smallskip}\dis
         p_t-(\sigma_i(x) p_x)_x + d(x,t)\,p + \int_0^t p(x,s)\, ds
          = h1_{\omega}         &  \mbox{in}&    (0,L) \times (0,T),      \\
             \noalign{\smallskip}\dis
             \sigma_i(x)p_x\big|_{x=0}=\sigma_i(x)p_x\big|_{x=L} = 0
        & \mbox{on}& (0,T),    \\
             \noalign{\smallskip}\dis
             p(\cdot,0) =p_0     & \mbox{in}&   (0,L).
         \end{array}
     \right.
\end{equation}
   Proposition \ref{nullcontrol} below will imply (global) null
controllability  for the heat equation with memory with Neumann boundary conditions.
   \end{rmq}

Exact controllability is a challenging topic which has received a tremendous amount of attention in the literature. We only comment on a few, closely related publications, { and in particular on the controllability  coupled PDE-ODE systems, and also the lack thereof.
We first note} that the linearized system \eqref{eq:monoNulllinear} above can be transformed into a parabolic equation with a distributed memory term, and an additional time-dependent  source term if $q_0\neq0$.
In \cite{brandao} optimal control of the monodomain equations is considered and  the question of approximate controllability of the heat equation with Dirichlet boundary conditions and   with memory term is raised as interesting open problem. In \cite{guerrero_imanu} it was verified that the heat equation with Dirichlet boundary conditions and memory terms is not exactly null controllable.   A related negative result was obtained in  \cite{breiten_kunisch} were it  was shown that the linearized monodomain system with Neumann boundary conditions is not exactly null controllable for a  range of physically relevant parameters.
{ In view of these negative results which refer to the case that the control domain is a fixed domain $\omega$ strictly contained in $\Omega$ it is natural to allow a more rich structure of the control mechanism. In particular a moving controller which during the time horizon $[0,T]$ covers all of $\Omega$. Such control mechanisms have first been introduced in \cite{martin_rosier_moving}, see also \cite{CSZR}.
In the recent publication \cite{zuazua_memory} exact null controllability of the linear heat equation with memory terms and Dirichlet boundary condition is obtained. There as well  the control domain appropriately   sweeps all the domain $\Omega$. The investigation of exact controllability for nonlinear systems with memory is pointed out as an interesting topic for further research. The authors of that paper mention that their technique is not applicable for the nonlinear case.}

	In this paper, we are going to treat different situations which lead to new difficulties compared to the previous
	works on parabolic equations with memory terms. Let us discuss these differences:
\begin{itemize}
	\item Neumann boundary condition. The main strategy consist of a Carleman inequality with suitable weights which allows to
	simultaneously treat the parabolic equation with Neumann boundary conditions and the ODE.
	The Carleman inequalities obtained in \cite{zuazua_memory} are inequalities
	for parabolic equations with Dirichlet boundary condition and cannot be applied for our situation.
	\item Nonlinearities. Here we are going to consider two kinds of (cubic) nonlinearities arising in the FitzHugh Nagumo and Rogers-McCulloch
	models which differ due to the way in which the memory term $w$ appears in the PDE equation, compare \eqref{eq:FN}
	and \eqref{eq:RM}.
	The main argument to obtain the exact controllability result for the nonlinear system relies
	on the Liusternik inverse mapping theorem.
	\item  Inhomogeneities.
	Another difficulty to verify exact null controllability arises due to the inhomogeneities  $G$ and $q_0$
	in the linear PDE-ODE system.
\end{itemize}

The paper is organized as follows.
	In Section~\ref{preliminaries_section}, we shall present Carleman inequalities with moving observations
	and regularity results for trajectories.
	Section~\ref{Section:monodomain} deals with the null controllability for linear monodomain equations and
	the local exact controllability to the trajectories for the nonlinear monodomain equations.
	Finally, in Appendix~\ref{App:mov_lapla}, we present a Carleman inequality for a parabolic
	equation with Neumann boundary conditions and moving observations.


\section{Preliminaries}\label{preliminaries_section}

    This section is devoted to introduce some appropriate Carleman inequalities with moving observations for ODE's and PDE's and
    regularity results for the trajectory solutions for \eqref{eq:monoTraj}.


\subsection{Carleman inequalities with moving observations}

    { In this section, we shall present the main assumptions on moving support of controls and observations and
    their consequences  in terms of Carleman inequalities. Many of them were introduced or  inspired by
    \cite{ALBANO,imanuvilov96,zuazua_memory,CSZR,CSSOUZA,BFCG,martin_rosier_moving}. These results will play a crucial role in order to obtain the exact controllability to trajectories for the monodomain equations.
 The control domain $\omega$ is required to contain  a subset $\omega_0$ which satisfies Assumption \ref{AssumptionMC}. It is  needed for the construction of a suitable weight function for the Carleman estimate.
\begin{figure}[h]
    \centering\includegraphics[scale=0.48]{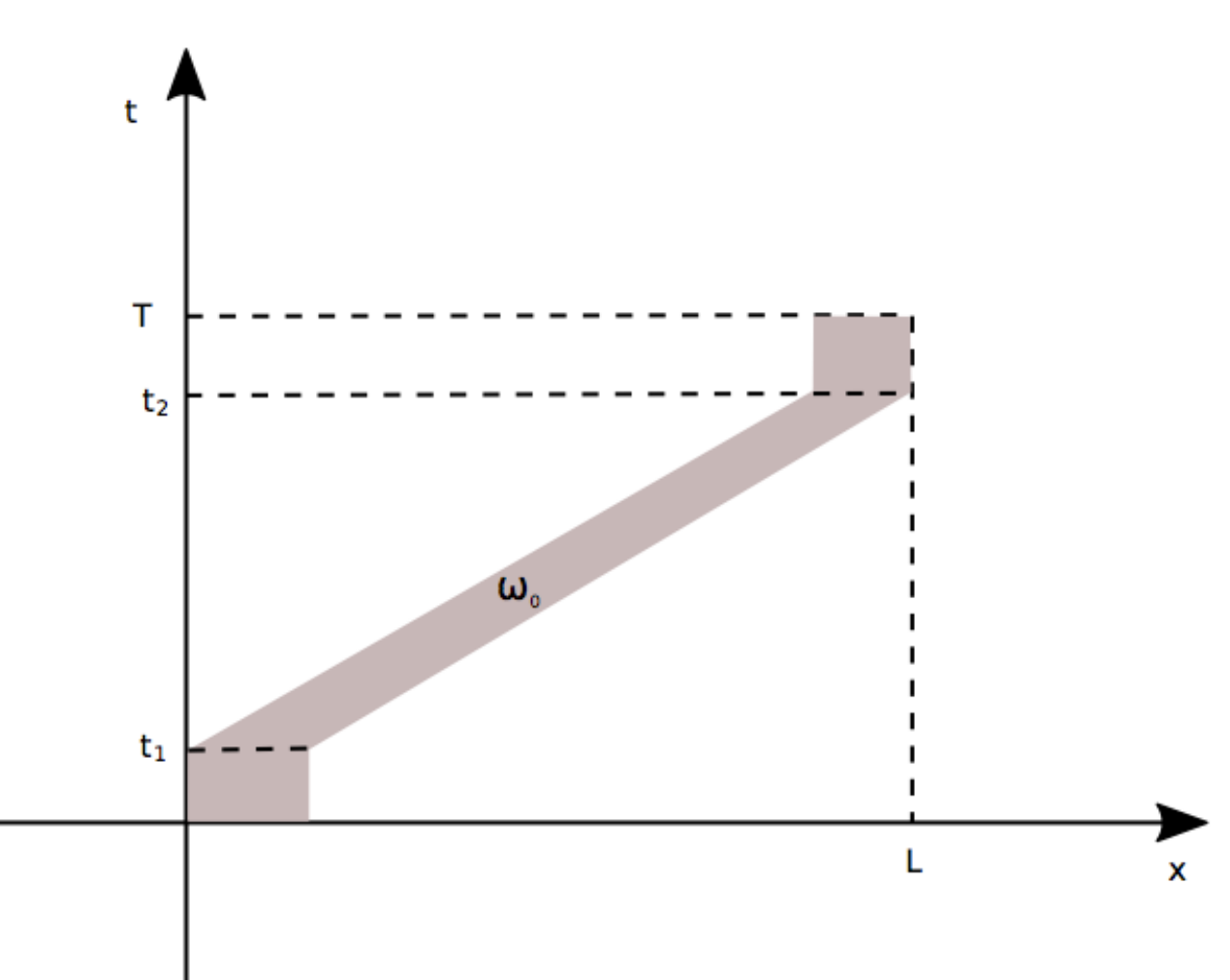}
    \caption{The control domain $\om_0$}\label{FIG1}
\end{figure}
}

	Let $\om_1:[0,T]\rightarrow 2^{(0,L)}$ be a subset of $\om$ such that
\begin{equation}\label{eq:om1om0}
	{\overline{\om}_0\subset
	\mathring{\overline{\om}}_1\quad\text{and}\quad\overline{\om}_1\subset
	\mathring{\overline{\om}}},
\end{equation}	
	where $\mathring{\overline{\om}}_1$ and $\mathring{\overline{\om}}$ denote the relative interiors with respect
	to $[0,L]\times[0,T]$ of ${\overline{\om}_1}$ and ${\overline{\om}}$, respectively.

    We are now prepared for the  construction of  a suitable weight function { which allow us to obtain Carleman inequalities
    for the parabolic equation coupled to the ordinary differential equation.}
\begin{lemma}\label{weight}
   There exist a positive number $\tau \in (0,\min\{1,T/2\})$, a positive
    constant $C_0>0$, and a function $\eta \in C^\infty ( [0,L]\times [0,T])$ such that
\ba
    \eta_x(x,t) \ne 0
    &&\forall x\in \overline{(0,L)\setminus\om_0(t)},\ \forall t\in[0,T],        \label{P1}\\
    \eta_t (x,t) \ne 0
    &&\forall x\in \overline{(0,L)\setminus\om_0(t)},\ \forall t\in[0,T],          \label{P2}\\
    \eta_t (x,t)  >0
    &&\forall x\in \overline{(0,L)\setminus\om_0(t)},\ \forall t\in[0,\tau],        \label{P3}\\
    \eta_t  (x,t) <0
    &&\forall x\in \overline{(0,L)\setminus\om_0(t)},\ \forall t\in[T-\tau,T],    \label{P4}\\
    \eta_x(0,t) \geq C_0
    &&\forall t\in [0,T ],     \label{P7}\\
    \eta_x(L,t)\leq-C_0
    &&\forall t\in [0,T ],     \label{P8}\\
    \min\limits_{(x,t)\in[0,L]\times [0,T]}\{\eta (x,t)\} =\frac{3}{4}\|\eta \|_{L^\infty([0,L]\times [0,T])}.&&
                            \label{P5}
\ea
\end{lemma}
    The proof of Lemma \ref{weight} can by obtained by similar arguments as in  \cite[Appendix $A$]{CSZR}.
    It differs with respect to properties \eqref{P7} and \eqref{P8}. { In fact, one just has to check that
    the weight obtained in \cite[Appendix $A$]{CSZR} satisfies also properties \eqref{P7} and \eqref{P8}}

    Next, we introduce a real function  $r\in C^\infty(0,T)$, symmetric with respect to $t={T\over2}$, i.e.
    $r(t)=r(T-t)$ for any $t\in(0,T)$, and such that
\begin{equation}\label{eq:r}
    r(t) = \left\{
    \begin{array}{lcl}
    \noalign{\smallskip}\dis
    {1\over t} &\text{for}& 0<t\leq{\tau\over2}, \\
    \noalign{\smallskip}\dis
    \text{\rm strictly decreasing}&\text{for}& {\tau\over2}< t <\tau,\\
    \noalign{\smallskip}\dis
    1    &\text{for}&\tau \le t \le {T\over2}.
    \end{array}
    \right.
\end{equation}

    For these choices of $\eta$ and $r$ let us define the weights
\begin{equation}\label{alpha}
    \begin{array}{lr}
    \alpha (x,t) := r(t) (e^{2\lambda \|\eta \|_\infty} - e^{\lambda\eta (x,t)})~~\hbox{and}~~
    \xi (x,t) := r(t) e^{ \lambda \eta (x,t)}&\forall(x,t)\in (0,L)\times (0,T),
\end{array}
\end{equation}
    where $\lambda >0$ is a  sufficiently large parameter that will be chosen later. We have the following technical result.

\begin{lemma}
\label{lemma:ode}
    There exist positive real numbers $\lambda _1> 0$,  $s_1> 0$ and $C_1>0$ (depending on $L$ and $\om_1$) such that for all
    $\lambda \ge \lambda _1$, all $s\ge s_1$ and all $\varphi\in H^1(0,T;L^2(0,L))$, the following
    inequality  holds
\begin{equation}\label{carleman:ode}
    s\lambda^2\iiTL \!\xi|\varphi|^2e^{-2s\alpha}\,dxdt\le C_1\left(\iiTL |\varphi_t|^2e^{-2s\alpha}\, dxdt
    + s^2\lambda ^2\int_0^T\!\!\!\!\int_{\om_1(t)}\xi^2|\varphi|^2e^{-2s\alpha}\,dxdt\right).
\end{equation}
\end{lemma}
 The  proof of this result is obtained by similar arguments as in \cite[Appendix $C$]{CSZR} and relies on ideas from \cite{ALBANO}.

 For our purposes, we also need the following Carleman inequality for the heat equation with Neumann boundary conditions:
\begin{lemma}
\label{lemma:elliptic}
   There exist constants  $\lambda _2>0$,  $s_2>0$ and
    $C_2>0$ (depending on $L$ and $\om_1$) such that  for any $\lambda \ge \lambda _2$, any
    $s\ge s_2(\lambda)$, and any terminal datum $\psi_T\in L^2(0,L)$, and any source term $f\in L^2((0,L)\times(0,T))$, the unique weak solution for
\[
    \left \{
        \begin{array}{lcl}
            -\psi_t-(\sigma_i(x) \psi_x)_x =f &\mbox{in}&    (0,L)\times(0,T),\\
            \noalign{\smallskip}\dis
\sigma_i(x)\psi_x\big|_{x=0}=\sigma_i(x)\psi_x\big|_{x=L} =0 &\mbox{in}&    (0,T),\\
            \noalign{\smallskip}\dis
            \psi(T) = \psi_T &\mbox{in}&    (0,L),
        \end{array}
    \right.
\]
    satisfies
\begin{equation}\label{carleman:neumann}
\begin{alignedat}{2}
    &s^{-1}\!\!\iiTL \!\!\!\xi^{-1}(|\psi_{xx}|^2+|\psi_t|^2)e^{-2s\alpha}\,dxdt\\
    &+
    s\lambda^2\!\!\iiTL\!\!\! \xi|\psi_x|^2e^{-2s\alpha}\,dxdt
    +s^3\lambda ^4\!\!\iiTL\!\!\! \xi^3|\psi|^2e^{-2s\alpha}
    \,dxdt\\
    &+s^3\lambda^3\int_0^T(\xi^3 |\psi|^2e^{-2s\alpha})\big|_{x=L}
    +s^3\lambda^3\int_0^T(\xi^3|\psi|^2e^{-2s\alpha})\big|_{x=0}\\
    &\le C_2 \left(\iiTL |f|^2e^{-2s \alpha}\,dxdt
    +s^3\lambda ^4\int_0^T\!\!\!\int_{\om_1(t)}
    \xi^3|\psi|^2e^{-2s\alpha}\,dxdt\right).     \end{alignedat}
\end{equation}
\end{lemma}
    We will give a proof for this result in Appendix \ref{App:mov_lapla}.


\subsection{Regular trajectories}

    In this section let us present regularity results for the uncontrolled solutions to \eqref{eq:monoTraj}.
    We have the following result:
\begin{propo}\label{propo:regular_trajec}
    Let $\mathcal{I}_{s,i}\in L^2(0,T;L^2(0,L))$ and $(\bar v_0,\bar w_0)\in H^1(0,L)\times H^2_\nu(0,L)$. Then, \eqref{eq:monoTraj}
    possesses exactly one solution $(\bar v,\bar w)$, with
\begin{equation}\label{reg_traj}
\left\{
    \begin{alignedat}{2}
        &\bar v\in L^2(0,T;H^2_\nu(0,L))\cap W^{1,2}(0,T;L^2(0,L))\\
        &\bar w\in  W^{1,2}(0,T;H^2_\nu(0,L))\cap W^{2,2}(0,T;L^2(0,L)).
    \end{alignedat}
    \right.
\end{equation}
\end{propo}
    For the proof, we just have to adapt the proof of \cite[Theorem $1.1$]{brandao}, taking into account that
    $\bar w_0\not\equiv0$ and the boundary conditions are of Neumann kind instead of Dirichlet. In this way, we can also obtain the following result.
\begin{propo}\label{propo:regular_trajec2}
    Let $\mathcal{I}_{s,i}\in H^1(0,T;H^1(0,L))$ and $(\bar v_0,\bar w_0)\in H^3_\nu(0,L)\times H^2_\nu(0,L)$. Then, \eqref{eq:monoTraj}
    possesses exactly one solution $(\bar v,\bar w)$, with
    \[
\left\{
    \begin{alignedat}{2}
         &\bar v\in L^\infty(0,T; H^3_\nu(0,L))\cap W^{1,2}(0,T;H^2_\nu(0,L))\cap W^{2,2}(0,T;L^2(0,L))\\
        &\bar w\in  W^{2,2}(0,T;H^2_\nu(0,L))\cap W^{3,2}(0,T;L^2(0,L)).
    \end{alignedat}
    \right.
\]
  \end{propo}
\begin{proof}
    By  Proposition \ref{propo:regular_trajec} and that the spatial dimension is 1, the nonlinearity $I_{ion}(\bar v,\bar w)$
    belongs to the space $W^{1,2}(0,T;L^2(0,L))\cap L^2(0,T;H^2_\nu(0,L))$. Hence $\mathcal{I}_{s,i}-I_{ion}(\bar v,\bar w)$ belongs to $H^1(0,T;H^1(0,L))$.
    Since $\bar v_0\in H^2_\nu(0,L)$, thanks to \cite[Theorem $5$, p. $361$]{evans}, we have
\begin{equation}\label{reg_traj_2}
\left\{
    \begin{alignedat}{2}
        &\bar v\in L^2(0,T;H^3_\nu(0,L))\cap W^{1,2}(0,T;H^1(0,L))\cap W^{2,2}(0,T;(H^1(0,L))')\\
        &\bar w\in  W^{1,\infty}(0,T;H^2_\nu(0,L))\cap W^{2,2}(0,T;H^1(0,L))\cap W^{3,2}(0,T;(H^1(0,L))').
    \end{alignedat}
    \right.
\end{equation}

    Next, taking  the spatial derivative in the first equation of \eqref{eq:monoTraj}  we obtain
    a parabolic equation for $\bar v_x$ with right hand side $\partial_x(\mathcal{I}_{s,i}-I_{ion}(\bar v,\bar w))$, homogeneous Dirichlet boundary
    conditions, and initial condition $\bar v_{0,x}\in H^2(0,L)\cap H^1_0(0,L)$. By \eqref{reg_traj_2} we have that
     $\partial_x(\mathcal{I}_{s,i}-I_{ion}(\bar v,\bar w))$ belongs to $W^{1,2}(0,T;L^2(0,L))$. Thanks to \cite[Theorem $5$, p. $361$]{evans}, we have
\begin{equation}\label{reg_traj_3}
    \begin{alignedat}{2}
        &\bar v_x\in L^\infty(0,T; H^2(0,L)\cap H^1_0(0,L))\cap W^{1,2}(0,T;H^1_0(0,L))\cap W^{2,2}(0,T;H^{-1}(0,L)).
    \end{alignedat}
\end{equation}

    Finally, from \eqref{reg_traj_2} and \eqref{reg_traj_3}, we deduce that
\begin{equation}\label{reg_traj_4}
\left\{
    \begin{alignedat}{2}
         &\bar v\in L^\infty(0,T; H^3_\nu(0,L))\cap W^{1,2}(0,T;H^2_\nu(0,L))\cap W^{2,2}(0,T;L^2(0,L))\\
        &\bar w\in  W^{2,2}(0,T;H^2_\nu(0,L))\cap W^{3,2}(0,T;L^2(0,L)).
    \end{alignedat}
    \right.
\end{equation}
\end{proof}

\begin{rmq}\label{rmk:lplq}
	As a consequence of the previous result we conclude that
$$
	\ell_p(\bar v,\bar w),\ell_q(\bar v,\bar w)\in  W^{1,2}(0,T;H^2_\nu(0,L))\cap W^{2,2}(0,T;L^2(0,L))\hookrightarrow
	W^{1,\infty}(0,T;L^\infty(0,L)).
$$
	{ This is needed to have bounded coefficients for the adjoint system \eqref{adj:mono}.}
\end{rmq}


 \section{Exact controllability to  trajectories of the monodomain equations}\label{Section:monodomain}

  This section is devoted to exact controllability results for linear and nonlinear monodomain equations.


 \subsection{Controllability for the linear monodomain model}\label{Section:linearmonodomain}

    In this section, we will present a suitable Carleman inequality for a properly chosen adjoint system.
    { This will allow us  to verify the null controllability result for the linearized system  \eqref{eq:monoNulllinear} in Proposition \ref{nullcontrol} below. This result will be essential for that verification of local null controllability of nonlinear system \eqref{eq:monoNull1}.}

    The following analysis depends in a crucial manner on a transformation of \eqref{eq:monoNulllinear}  which takes the  control from the PDE to the ODE. { It is based on the change of variables $\theta =q_t-(\sigma_i(x)q_x)_x+\ell_p(\bar v,\bar w)q$}.
 We will verify that null controllability for  \eqref{eq:monoNulllinear} is equivalent to
    null controllability for
\begin{equation}\label{eq:monoNulllinear1}
    \!\!\!\!\!\left\{
        \begin{array}{lcl}
            \theta_t+\{\ell_q(\bar v,\bar w)-[\ell_p(\bar v,\bar w)]_t\} q
         = G+h1_{\omega}         &  \mbox{in}&    (0,L) \times (0,T),        \\
            \noalign{\smallskip}\dis
            q_t-(\sigma_i(x)q_x)_x+\ell_p(\bar v,\bar w)q=\theta &\mbox{in}&(0,L) \times (0,T),    \\
            \noalign{\smallskip}\dis
            \sigma_i(x)q_x\big|_{x=0}=\sigma_i(x)q_x\big|_{x=L} = 0                & \mbox{on}& (0,T),    \\
            \noalign{\smallskip}\dis
            \theta(\cdot,0) =\theta_0,\quad q(\cdot,0) =q_0 & \mbox{in}&    (0,L),
        \end{array}
    \right.
\end{equation}
    where $\theta_0=p_0-(\sigma_i (x)q_{0,x})_x +\ell_p(\bar v_0,\bar w_0)q_0$.

	First we verify the equivalence of systems \eqref{eq:monoNulllinear} and \eqref{eq:monoNulllinear1}.
Notice that, for every $(p_0,q_0)\in L^2(0,L)\times H^1(0,L)$ and every
	$G,h1_{\omega}\in L^2((0,L)\times(0,T))$,
	there exists a unique weak solution $(p,q)$ to~\eqref{eq:monoNulllinear} that satisfies
\[
	p\in L^2(0,T;H^1(0,L))\cap W^{1,2}(0,T;(H^1(0,L))^\prime)
\]
	and
\[
	q\in  W^{1,2}(0,T;H^1(0,L))\cap W^{2,2}(0,T;(H^1(0,L))^\prime)
\]
	and the variational formulation for a.e. $t\in[0,T]$
\begin{equation}\label{var:pq}
	\left\{
		\begin{array}{l}
	
	\langle p_t,w\rangle+(\sigma p_x, w_x) +(\ell_p(\bar v,\bar w)p,w)
	+ (\ell_q(\bar v,\bar w) q,w)=(G+h1_\om,w),~\forall w\in H^1(0,L),\\
	\noalign{\smallskip}\dis
	(q_t,v)=(p,v),~\forall v\in L^2(0,L),\\
	\noalign{\smallskip}\dis
	p(\cdot,0)=p_0 ~\hbox{and}~q(\cdot,0)=q_0,
\end{array}
	\right.
\end{equation}
	where $\langle \cdot,\cdot\rangle:=\langle \cdot,\cdot\rangle_{(H^1(0,L))^\prime,H^1(0,L)}$.\newline
	For system~\eqref{eq:monoNulllinear1}, there exists a unique weak solution $(\theta,q)$ that satisfies
\[
	\theta \in  W^{1,2}(0,T;(H^1(0,L))^\prime)
\]
	and
\[
	q \in  W^{1,2}(0,T;H^1(0,L))\cap W^{2,2}(0,T;(H^1(0,L))^\prime)
\]
   	and the variational formulation for a.e. $t\in[0,T]$
\begin{equation}\label{var:thetaq}
	\left\{
		\begin{array}{l}
		\langle\theta_t,v\rangle+(\{\ell_q(\bar v,\bar w)-[\ell_p(\bar v,\bar w)]_t\} q,v)=(G+h1_{\omega} ,v),~\forall v\in H^1(0,L)\\
	\noalign{\smallskip}\dis
	\langle q_t,w\rangle+(\sigma q_x, w_x) +(\ell_p(\bar v,\bar w)q,w)=
	\langle \theta,w\rangle,~\forall w\in H^1(0,L)\\
	\noalign{\smallskip}\dis
	\theta(\cdot,0)=\theta_0 ~\hbox{and}~q(\cdot,0)=q_0,
		\end{array}
	\right.
\end{equation}
	where again $\theta_0=p_0-(\sigma_i (x)q_{0,x})_x +\ell_p(\bar v_0,\bar w_0)q_0$.

To check the equivalence one can start by defining $\theta \in  W^{1,2}(0,T;(H^1(0,L))^\prime)$ such that
$$
	\left\{
		\begin{array}{lcl}
	 \theta_t  + \{\ell_q(\bar v,\bar w)-[\ell_p(\bar v,\bar w)]_t\}q= G+ h1_{\omega}         & \mbox{in}&    (0,L) \times (0,T),        \\
            \noalign{\smallskip}\dis
             \theta(\cdot,0) =\theta_0& \mbox{in}&    (0,L),
		\end{array}
	\right.
$$
	Then for each $w\in H^1(0,L)$ and integrating with respect to $t$ the first equation in \eqref{var:pq}
$$
\begin{alignedat}{2}
	\langle q_t,w\rangle+(\sigma q_x, w_x)=&
\left(\int_0^t[G+h1_\om-\ell_p(\bar v,\bar w)p-\ell_q(\bar v,\bar w) q]ds,w\right)\\
&+(p_0,w)+(\sigma q_{0,x}, w_x)\\
=&
\bigg\langle\int_0^t\theta_tds,w\bigg\rangle-\left(\int_0^t(\ell_p(\bar v,\bar w)q)_tds,w\right)+(p_0,w)+(\sigma q_{0,x}, w_x)\\
=&\langle\theta,w\rangle+(\ell_p(\bar v,\bar w)q,w)-\langle\theta(0)-\theta_0,w\rangle\\
=&\langle\theta,w\rangle+(\ell_p(\bar v,\bar w)q,w).
\end{alignedat}
$$
	Thus $(\theta,q)$ is the variational solution of \eqref{var:thetaq}.

	Conversely, let $(p_0,q_0)\in L^2(0,L)\times H^1(0,L)$ then $\theta_0=p_0-(\sigma_i (x)q_{0,x})_x +\ell_p(\bar v_0,\bar w_0)q_0\in (H^1(0,L))^\prime$ and
\[
	\theta \in  W^{1,2}(0,T;(H^1(0,L))^\prime)
\]
	and
\[
	q \in  W^{1,2}(0,T;H^1(0,L))\cap W^{2,2}(0,T;(H^1(0,L))^\prime),
\]
	as in \eqref{var:thetaq}.
	
	Regularity allows to differentiate the second equation in \eqref{var:thetaq} with respect to $t$ and, setting $p=q_t$ we obtain
$$
\begin{alignedat}{2}
		\langle p_t,w\rangle+(\sigma p_x, w_x) +(\ell_p(\bar v,\bar w)p,w)=&~
	\langle \theta_t,w\rangle-(\{\ell_p(\bar v,\bar w)\}_tq,w)\\
=&~-(\ell_q(\bar v,\bar w) q,w)+(G+h1_{\omega} ,w).
\end{alignedat}
$$
	Thus $(p,q)$ is the variational solution of \eqref{var:pq}.

\

    Let us introduce the function $A(x,t):=\{\ell_q(\bar v,\bar w)-[\ell_p(\bar v,\bar w)]_t\}$ and the non-autonomous
    elliptic operator $Kq:=-(\sigma_i(x)q_x)_x+\ell_p(\bar v,\bar w)q$. Then, the controllability for system \eqref{eq:monoNulllinear1}
    is formulated as follows: for any $(\theta_0,q_0)\in L^2(0,L)\times L^2(0,L)$ and any $G$ satisfying \eqref{eq:G} below,
    there exists a control $h\in L^2((0,L)\times(0,T))$, where its support contains a subset 
    satisfying Assumption \ref{AssumptionMC}, such that the
    associated solution to
\begin{equation}\label{eq:mono2}
    \left\{
        \begin{array}{lcl}
            \theta_t  + A(x,t)\,q= G+ h1_{\omega}         & \mbox{in}&    (0,L) \times (0,T),        \\
            \noalign{\smallskip}\dis
            q_t+Kq  =  \theta                    & \mbox{in}&    (0,L) \times (0,T),        \\
            \noalign{\smallskip}\dis
            \sigma_i(x)q_{x}\big|_{x=0}=\sigma_i(x)q_{x}\big|_{x=L} = 0                & \mbox{on}&     (0,T),    \\
            \noalign{\smallskip}\dis
            \theta(\cdot,0) =\theta_0,\quad q(\cdot,0) =q_0 & \mbox{in}&    (0,L),
        \end{array}
    \right.
\end{equation}
 satisfies
 $$
     \theta(\cdot,T) =q(\cdot,T)=0\quad \hbox{in}\quad (0,L).
 $$

\begin{rmq}\label{rmq_A}
	{ As we will see the null controllability for system \eqref{eq:mono2} will follow from a  Carleman inequality for its
	adjoint. In order to prove this  Carleman inequality we will combine Lemmas \ref{lemma:ode} and \ref{lemma:elliptic}. Notice also that  thanks to the Remark \ref{rmk:lplq}, $A\in L^\infty((0,L)\times(0,T))$. This will be  used for Theorem \ref{thm:Carl_mono} to 
	prove the Carleman inequality.}
\end{rmq}


\subsubsection{Carleman inequality}\label{Carleman}

	In this Section, we will present a suitable Carleman inequality for the so called adjoint of \eqref{eq:mono2},
	namely:
\begin{equation}\label{adj:mono}
	\left\{
		\begin{array}{lcl}
			- \varphi_t  -\psi=R 					&\mbox{in}&	 (0,L) \times (0,T), 		\\
			\noalign{\smallskip}\dis
			-\psi_t+  K\psi+A(x,t)\varphi=S 	&\mbox{in}&	 (0,L) \times (0,T),  		\\
			\noalign{\smallskip}\dis
			\sigma_i(x)\psi_{x}\big|_{x=0}=\sigma_i(x)\psi_{x}\big|_{x=L} = 0     				&\mbox{on}&	(0,T), 	\\
			\noalign{\smallskip}\dis
			\varphi(T) = \varphi_T,\quad\psi(T) = \psi_T 		&\mbox{in}&	(0,L),
		\end{array}
	\right.
\end{equation}
	where $R,S\in L^2((0,L)\times(0,T))$.
	
	Note that, for every $( \varphi_T,\psi_T )\in L^2(0,L)\times L^2(0,L)$ and every
	$R,S\in L^2((0,L)\times(0,T))$
	there exists a unique weak solution $(\varphi,\psi)$ to~\eqref{adj:mono} that satisfies
\[
	\varphi \in H^1(0,T;L^2(0,L))
	\quad\hbox{and}\quad
	\psi\in  L^2(0,T;H^1(0,L))\cap W^{1,2}(0,T;(H^1(0,L))')
\]
	and the variational formulation
$$
	\left\{
		\begin{array}{l}
	-(\varphi_t,v)-(\psi, v)=(R,v)~\forall v\in L^2(0,L),~\hbox{a.e.}~ t\in[0,T]\\
	\noalign{\smallskip}\dis
	-\langle\psi_t,w\rangle+(\sigma\psi_x, w_x) +(\ell_p(\bar v,\bar w)\psi,w)
	+ (A\varphi,w)=(S,w)~\forall w\in H^1(\Om),~\hbox{a.e.}~ t\in[0,T]\\
	\noalign{\smallskip}\dis
	\varphi(\cdot,T)=\varphi_T ~\hbox{and}~\psi(\cdot,T)=\psi_T.
		\end{array}
	\right.
$$
	
	In order to prove the controllability result for the linear system given by Proposition \ref{nullcontrol}, we have to establish an appropriate Carleman estimate.
	Precisely, we have the following result:
\begin{thm}\label{thm:Carl_mono}
	There exist positive constants $s_3$, $\lambda_3\geq 1$ and~$C_3>0$, only depending on $L$ and $\omega$, such that,
	for any $\varphi_T,\psi _T\in L^2(0,L)$ the solution $(\psi,\varphi)$ to adjoint system \eqref{adj:mono} satisfies:
\begin{align}
		&\iiTL[(s\xi)^{-1}(|\psi_{xx}|^2+|\psi_t|^2)+\lambda^2(s\xi)|\psi_x|^2+\lambda ^4(s\xi)^3|\psi|^2+\lambda ^2 (s\xi)   |\varphi |^2 ]e^{-2s\alpha}dxdt\nonumber\\
		\le &~C_3
		\left(\iiTL\![\lambda ^4(s\xi)^3|R|^2+|S|^2] e^{-2s\alpha} dxdt +
		 { s^7\lambda^8}\int_0^T \!\!\!\int_{\om(t)} \xi ^7  |\varphi|^2e^{-2s\alpha}  dxdt\right).\nonumber
	\end{align}
	for all  $s \geq  s_3(T + T^2) $ and for all $\lambda \geq \lambda_3$.
\end{thm}
\begin{proof}
	First, applying Lemma \ref{lemma:ode} to $\eqref{adj:mono}_1$, we obtain
\be  \label{est:ode}
	\!\!\!s \lambda ^2 \iiTL\!\!\!\xi  |  \varphi |^2 e^{-2s\alpha}dxdt \le C_1
	\left(\iiTL\!\!\!|\psi +R|^2 e^{-2s\alpha} dxdt
	+s^2 \lambda ^2 \int_0^T\!\!\!\int_{\om_1(t)}\!\!\!\xi ^2  |\varphi|^2e^{-2s\alpha} dxdt\right)\!.
\ee
	
	Next, applying Lemma \ref{lemma:elliptic} to $\eqref{adj:mono}_2$, we obtain
\begin{equation}\label{est:elliptic}
\begin{alignedat}{2}
	&\iiTL[(s\xi)^{-1}(|\psi_{xx}|^2+|\psi_t|^2)+\lambda^2(s\xi)|\psi_x|^2+\lambda ^4(s\xi)^3|\psi|^2]
	e^{-2s\alpha}\,dx\,dt\\
	&\le C_2 \left(\iiTL|  S-\ell_p(\bar v,\bar w)\psi-A(x,t)\varphi|^2 e^{-2s \alpha}\,dx\,dt
	+\int_0^T\!\!\!\int_{\om_1(t)}
	\lambda ^4(s\xi)^3|\psi|^2e^{-2s\alpha}\,dx\,dt\right). 	
\end{alignedat}
\end{equation}

Adding \eqref{est:ode} and  \eqref{est:elliptic}, and absorbing the lower order terms from the right hand side by taking $\lambda$ large enough,  we get
\begin{equation}\label{est:elliptic+ode}
\begin{alignedat}{2}
	&\iiTL[((s\xi)^{-1}(|\psi_{xx}|^2+|\psi_t|^2)+\lambda^2(s\xi)|\psi_x|^2+\lambda ^4(s\xi)^3|\psi|^2+\lambda ^2 (s\xi)   |\varphi |^2 ]e^{-2s\alpha}dxdt\\
		&\le C \left(\iiTL\!(|R|^2+|S|^2) e^{-2s\alpha} dxdt +
		\int_0^T \!\!\! \int_{\om_1(t)}[\lambda ^2(s\xi)^2 |\varphi|^2 + \lambda ^4 (s\xi)^3 |\psi|^2] e^{-2s\alpha} dxdt \right).
	\end{alignedat}
\end{equation}

	Now, we need to eliminate the local integral of $\psi$ appearing in the right hand side of \eqref{est:elliptic+ode}.
	To do that, let us introduce a function $\zeta \in C^\infty([0,L]\times[0,T])$ satisfying
\begin{equation}\label{zeta}
	\begin{array}{lll}
	&0\le \zeta\le 1&\forall (x,t)\in [0,L]\times[0,T],\\
	\noalign{\smallskip}\dis
	&\zeta(x,t) =1&\forall x\in\om_1(t),~\forall t\in[0,T],\\
	\noalign{\smallskip}\dis
	&\zeta(x,t) =0& \forall x\in [0,L]\setminus \overline{\om(t)},~ \forall t\in[0,T].
	\end{array}
\end{equation}

	In this way, we have that
\begin{equation}\label{eq:cut-offBZK}
	s^3 \lambda ^4\int_0^T \!\!\! \int_{\om_1(t)}\xi ^3 |\psi|^2 e^{-2s\alpha} dxdt
	\leq s^3\lambda ^4\iiTL\zeta \xi ^3 |\psi|^2 e^{-2s\alpha} dxdt .
\end{equation}
	Using $ \eqref{adj:mono}_1$ and integrating by parts with respect to  $t$, we obtain
\begin{align*}
	s^3\lambda ^4\iiTL\zeta \xi ^3 |\psi|^2 e^{-2s\alpha} dxdt
	=& -s^3\lambda ^4 \iiTL \zeta\xi ^3\psi(\varphi_t+R) e^{-2s\alpha}  dxdt\\
	=& \,s^3\lambda ^4 \iiTL\zeta\xi^3  \psi_t\varphi e^{-2s\alpha}  dxdt
	 +  s^3\lambda ^4\iiTL \zeta_t\xi ^3\varphi\psi e^{-2s\alpha}  dxdt  \\
	&+s^3\lambda ^4 \iiTL\zeta\xi^2(3\xi_t - 2 s\xi\alpha_t) \varphi\psi e^{-2s\alpha}  dxdt\\
	& -s^3\lambda ^4 \iiTL \zeta\xi ^3\psi R e^{-2s\alpha}  dxdt\\
   	:=&\,D_1+D_2+D_3+D_4.
\end{align*}
	For $D_1$, we notice that for every $\epsilon >0$, we obtain
\begin{equation}\label{est:A21}
	\begin{alignedat}{2}
		|D_1| \le&~ \epsilon  \iiTL(s\xi)^{-1}  |\psi_t|^2 e^{-2s\alpha} dxdt
		+C_{\epsilon}s^7\lambda^8 \iiTL \zeta^2\xi^7  |\varphi|^2
		e^{-2s\alpha}  dxdt\\
		\le&~ \epsilon  \iiTL(s\xi)^{-1}  |\psi_t|^2 e^{-2s\alpha} dxdt
		+C_{\epsilon}s^7\lambda^8 \int_0^T \!\!\! \int_{\om(t)}\xi^7  |\varphi|^2
		e^{-2s\alpha}  dxdt.
	\end{alignedat}
\end{equation}
	Next, $D_2$ is estimated like the term $D_1$:
\begin{equation}\label{est:A23}
	|D_2| \leq \epsilon s^3\lambda^4  \iiTL \xi ^3 |\psi|^2  e^{-2s\alpha} dxdt
	+ C_{\epsilon} s^3\lambda^4 \int_0^T \!\!\! \int_{\om(t)}\xi ^3 |\varphi|^2 e^{-2s\alpha} dxdt.
\end{equation}
	for any $\epsilon >0$.
	Since $|\xi_t| +|\alpha_t| \leq C \lambda\xi  ^2$, for every $\epsilon >0$, we infer that
\begin{equation}\label{est:A22}
	\begin{alignedat}{2}
	|D_3| \le&~  Cs^4\lambda^5 \iiTL \zeta    \xi^5  |\varphi\psi| e^{-2s\alpha} dxdt \nonumber  \\
	\le&~\epsilon s^3\lambda^4\iiTL\xi ^{3}|\psi|^2 e^{-2s\alpha} dxdt
	+C_\epsilon s^5\lambda^6\int_0^T \!\!\! \int_{\om(t)} \xi ^7  |\varphi|^2e^{-2s\alpha}  dxdt .
	\end{alignedat}
\end{equation}
		Finally, $D_4$ is estimated as follows:
\begin{equation}\label{est:A24}
	|D_4| \leq \epsilon s^3\lambda^4  \iiTL \xi ^3 |\psi|^2  e^{-2s\alpha} dxdt
	+ C_{\epsilon} s^3\lambda^4 \int_0^T \!\!\! \int_{0}^{L}\xi ^3 |R|^2 e^{-2s\alpha} dxdt,
\end{equation}
	for any $\epsilon >0$.
Combining \eqref{est:elliptic+ode} and  \eqref{eq:cut-offBZK}-\eqref{est:A24}, and absorbing the lower order terms,  we get
\begin{equation}\label{est:finalmono}
\begin{alignedat}{2}
	&\iiTL[((s\xi)^{-1}(|\psi_{xx}|^2+|\psi_t|^2)+\lambda^2(s\xi)|\psi_x|^2+\lambda ^4(s\xi)^3|\psi|^2+\lambda ^2 (s\xi)   |\varphi |^2 ]e^{-2s\alpha}dxdt\\
		&\le C_3  \left(\iiTL\![\lambda ^4(s\xi)^3|R|^2+|S|^2] e^{-2s\alpha} dxdt +
		 s^7\lambda^6\int_0^T \!\!\!\int_{\om(t)} \xi ^7  |\varphi|^2e^{-2s\alpha}  dxdt\right).
	\end{alignedat}
\end{equation}
\end{proof}


\subsubsection{Null controllability}

	{ In this section we  prove the null controllability property for system \eqref{eq:monoNulllinear}
	with a right hand side which decays exponentially as $t\rightarrow T^-$. As we will see in the next section,
	this result will be useful to deduce the local null controllability of \eqref{eq:monoNull1}. First, however,  we need to deduce from Theorem~\ref{thm:Carl_mono} a second Carleman inequality with weights that
	do not vanish at $t = 0$}.

	More precisely, let us consider the function
\begin{equation*}
	l(t)=
\left\{
\begin{array}{lll}
	1				&\text{for}& 0  \leq t\leq T/2,\\
	r(t)				&\text{for}& T/2\leq t\leq T,
\end{array}
\right.
\end{equation*}
	with $r$ defined in \eqref{eq:r},
	and the following associated weight functions:
\[
	\begin{array}{ll}
	\bar\alpha(x,t):= l(t) (e^{2\lambda \|\eta \|_\infty} - e^{\lambda\eta (x,t)})&
	\forall (x,t)\in (0,L)\times(0,T),\\
	\noalign{\smallskip}\dis
	\bar\xi(x,t) := l(t) e^{ \lambda\eta (x,t)}&\forall (x,t)\in (0,L)\times(0,T),\\
	\noalign{\smallskip}\dis
	\bar\alpha^*(t):= \min\limits_{x\in [0,L]}\bar\alpha(x,t)=l(t) (e^{2\lambda \|\eta \|_\infty} - e^{\lambda\|\eta (\cdot,t)\|_\infty}),&
	\forall t\in (0,T),\\
	\noalign{\smallskip}\dis
	\bar\xi^*(t) := \max\limits_{x\in [0,L]}\bar\xi(x,t)=l(t) e^{ \lambda \|\eta (\cdot,t)\|_\infty}&\forall t\in (0,T),\\
	\noalign{\smallskip}\dis
	\hat{\bar\alpha}(t):= \max\limits_{x\in [0,L]}\bar\alpha(x,t),&
	\forall t\in (0,T),\\
	\noalign{\smallskip}\dis
		\hat{\bar\xi}(t) := \min\limits_{x\in [0,L]}\bar\xi(x,t)&\forall t\in (0,T),
\end{array}
\]	
	and
\begin{equation*}
\begin{array}{llll}
	\noalign{\smallskip}\dis
	\rho_1(t)=(\bar\xi^*)^{-3/2}e^{s\bar\alpha^*},
	~\, \rho_2(t)= e^{s\bar\alpha^*},
	~\, \rho_3(t)=(\bar\xi^*)^{-7/2}e^{s\bar\alpha^*},
	~\, \rho_4(t)=(\hat{\bar\xi})^{-1/2} e^{s\hat{\bar\alpha}}.
\end{array}
\end{equation*}

\begin{rmq}\label{rmq:infty}
	Notice that $\bar\alpha^*$ and $\hat{\bar\alpha}$ blow up exponentially as $t\to T^-$ and
$\bar\xi^*$ and $\hat{\bar\xi}$  blow up polynomially as $t\to T^-$.
\end{rmq}

	By combining Theorem~\ref{thm:Carl_mono} and classical energy estimates satisfied by $\varphi$ and $\psi$,
	we easily deduce the following result:
\begin{lemma}\label{theo:linearnullcontrol}
	Under the assumptions of Theorem~\emph{\ref{teo1}}, there exist positive constants
	 $s_4$, $\lambda_4\geq 1$ and~$C_4>0$, only depending on $L$ and $\omega$ such that, for any
	$\varphi _T,\psi_T \in L^2(0,L)$ and any $R,S\in L^2((0,L)\times(0,T))$, the solution to the adjoint system \eqref{adj:mono} 		satisfies:
\begin{equation}\label{carleman_0}
	\begin{alignedat}{2}
	&\!\!\!\!\!\iiTL\![{\bar\xi}^{-1}(|\psi_{xx}|^2\!+\!|\psi_t|^2)\!+\!\bar\xi|\psi_x|^2
		\!+\!\bar\xi^3|\psi|^2\!+\!\bar\xi  |\varphi |^2 ]e^{-2s\bar\alpha}dxdt
		+ \!\|\varphi(\cdot,0)\|^2_{L^2(0,L)}\!+\!\|\psi(\cdot,0)\|^2_{L^2(0,L)}\\
		&\le C_4
		\left(\iiTL\![\rho_1^{-2}|R|^2+\rho_2^{-2}|S|^2] dxdt +
		 \int_0^T \!\!\!\int_{\om(t)} \rho_3^{-2}|\varphi|^2  dxdt\right).
\end{alignedat}
\end{equation}
	for all  $s \geq  s_4(T + T^2) $ and for all $\lambda \geq \lambda_4$.
\end{lemma}
\begin{proof}
	Let us start by proving the following estimate for system \eqref{adj:mono}:
\begin{equation}\label{energy_est}
\begin{alignedat}{2}
		&\|\varphi\|^2_{ W^{1,2}(0,T/2;L^2(0,L))}+
		\|\psi\|^2_{L^2(0,T/2;H^2_\nu(0,L))\cap W^{1,2}(0,T/2;L^2(0,L))}\\
		\le &
		C\bigg(\|(R,S)\|^2_{L^2(0,3T/4;[L^2(0,L)]^2)}+{1\over T^2}\|(\varphi,\psi)\|^2_{L^2(T/2,3T/4;[L^2(0,L)]^2)}\bigg).
\end{alignedat}
\end{equation}
	To do that, let us introduce a function $\kappa\in C^1([0,T])$ with
$$
	\kappa\equiv1\quad\hbox{in}\quad [0,T/2],
	\qquad\kappa\equiv0\quad\hbox{in}\quad [3T/4,T],
	\qquad |\kappa'|\leq C/T,
$$
	for some $C>0$. Using classical energy estimates for the system satisfied by $(\kappa\varphi,\kappa\psi)$, which is
	analogous to \eqref{adj:mono}, we obtain
\begin{align*}
		&\|\kappa\varphi\|^2_{ W^{1,2}(0,T;L^2(0,L))}+
		\|\kappa\psi\|^2_{L^2(0,T;H^2_\nu(0,L))\cap W^{1,2}(0,T;L^2(0,L))}\\
		\le &
		C\bigg(\|(\kappa R,\kappa S)\|^2_{L^2(0,T;[L^2(0,L)]^2)}+\|(\kappa'\varphi,\kappa'\psi)\|^2_{L^2(0,T;[L^2(0,L)]^2)}\bigg),
\end{align*}
	which leads to \eqref{energy_est}.
	
	Since the weights are bounded from above and below, using \eqref{energy_est}, we obtain
	a first estimate in $(0,L)\times(0,T/2)$:
\begin{equation}\label{carleman_001}
\begin{alignedat}{2}
		&\!\!\!\!\!\!\!\!\!\int_0^{T\over2}\!\!\!\!\int_0^L\!\![{\bar\xi}^{-1}(|\psi_{xx}|^2\!+\!|\psi_t|^2)\!+\!\bar\xi|\psi_x|^2
		\!+\!\bar\xi^3|\psi|^2\!+\!\bar\xi|\varphi |^2 ]e^{-2s\bar\alpha}dxdt
		\!+\!\|\varphi(\cdot,0)\|^2_{L^2(0,L)}\!+\!\|\psi(\cdot,0)\|^2_{L^2(0,L)} \\
		\le &~C(T,s,\lambda)
		\left(\int_0^{3T\over4}\!\!\!\!\!\int_0^L\![\rho_1^{-2}|R|^2+\rho_2^{-2}|S|^2] dxdt +
		 \int_{T\over2}^{3T\over4}\!\!\!\!\!\int_0^L[\bar\xi^3|\psi|^2+\bar\xi|\varphi |^2 ]e^{-2s\bar\alpha}dxdt\right).
\end{alignedat}
\end{equation}	
	
	On the other hand, since $\alpha=\bar\alpha$ and $\xi=\bar\xi$ in $(0,L)\times(T/2,T)$, we have:
\begin{align}\label{carleman_002}
		&\int_{T\over2}^T\!\!\!\int_0^L[(s{\bar\xi})^{-1}(|\psi_{xx}|^2+|\psi_t|^2)+\lambda^2(s\bar\xi)|\psi_x|^2
		+\lambda ^4(s\bar\xi)^3|\psi|^2+\lambda ^2 (s\bar\xi)  |\varphi |^2 ]e^{-2s\bar\alpha}dxdt
		 \nonumber\\
		 =&\int_{T\over2}^T\!\!\!\int_0^L[(s{\xi})^{-1}(|\psi_{xx}|^2+|\psi_t|^2)+\lambda^2(s\xi)|\psi_x|^2
		+\lambda ^4(s\xi)^3|\psi|^2+\lambda ^2 (s\xi)  |\varphi |^2 ]e^{-2s\alpha}dxdt
		 \\
		\le &~\iiTL[((s\xi)^{-1}(|\psi_{xx}|^2+|\psi_t|^2)+\lambda^2(s\xi)|\psi_x|^2+\lambda ^4(s\xi)^3|\psi|^2+\lambda ^2 (s\xi)   |\varphi |^2 ]e^{-2s\alpha}dxdt.\nonumber
\end{align}	
	In this way we obtain by Theorem~\ref{thm:Carl_mono},
	\begin{align*}\label{carleman_003}
		&
		\int_{T\over2}^T\!\!\!\int_0^L[{\bar\xi}^{-1}(|\psi_{xx}|^2+|\psi_t|^2)+\bar\xi|\psi_x|^2
		+\bar\xi^3|\psi|^2+\lambda ^2 \bar\xi |\varphi |^2 ]e^{-2s\bar\alpha}dxdt
		 \\
		 \le &~C(T,s,\lambda)
		\left(\iiTL\![\xi^3|R|^2+|S|^2] e^{-2s\alpha} dxdt +
		 \int_0^T \!\!\!\int_{\om(t)} \xi ^7  |\varphi|^2e^{-2s\alpha}  dxdt\right).
\end{align*}	
	Finally, from the definition of $\bar\alpha$, $\bar\xi$, $\bar\alpha^*$ and $\bar\xi^*$, we deduce:
\begin{align*}
		&\int_{T\over2}^T\!\!\!\int_0^L[{\bar\xi}^{-1}(|\psi_{xx}|^2+|\psi_t|^2)+\bar\xi|\psi_x|^2
		+\bar\xi^3|\psi|^2+\lambda ^2 \bar\xi |\varphi |^2 ]e^{-2s\bar\alpha}dxdt
		 \\
		 \le &~C(T,s,\lambda)\left(\iiTL\![\rho_1^{-2}|R|^2+\rho_2^{-2}|S|^2] dxdt +
		 \int_0^T \!\!\!\int_{\om(t)} \rho_3^{-2}|\varphi|^2  dxdt\right),
\end{align*}
	which, combined with \eqref{carleman_001}, provides \eqref{carleman_0}.
\end{proof}

	\

	The next step is to prove null controllability of the linear system \eqref{eq:monoNulllinear}.
	Of course, we will need some specific conditions on the source $G$.
	Thus let us introduce the linear operators $\mathcal{M}_1$ and $\mathcal{M}_2$ by
\begin{equation}\label{Ls}
	\mathcal{M}_1(\theta)= \theta_{t}
	\quad
	\hbox{and}
	\quad
	\mathcal{M}_2(q)= q_{t}+ Kq,
\end{equation}
	and the spaces
\begin{align*}
	E_0
		=&~ \{\, (q,h): \rho_1\mathcal{M}_2(q), \rho_2 q,\rho_3h1_\om\in L^2((0,L)\times(0,T)),\\
		& \ \rho_4^{1/2} q \in W^{1,2}(0,T;L^2(0,L))\cap L^2(0,T;H^2_\nu(0,L))\cap L^{\infty}(0,T;H^1(0,L)),\\
		 & \ \rho_4^{1/2}\mathcal{M}_2(q)\in W^{1,2}(0,T;L^2(0,L))\,\}
\end{align*}
	and
\begin{align*}
	E  =&~ \{\, (q,h): (q,h) \in E_0,\,\,	 								
		\rho_4[\mathcal{M}_1(\mathcal{M}_2(q)) + Aq - h1_{\omega}]\in L^2((0,L)\times(0,T))\\
		& \ \rho_4^{1/2} \mathcal{M}_1(q) \in W^{1,2}(0,T;(H^1(0,L))')\cap  L^2(0,T;H^1(0,L))\cap  L^\infty(0,T;L^2(0,L)),\\
		& \ \rho_4^{1/2} \mathcal{M}_1(q) \in L^4((0,L)\times(0,T)), \rho_4^{1/3} \mathcal{M}_1(q) \in L^6((0,L)\times(0,T)),
		\,  q(\cdot,0)\in H^2_\nu(0,L)\,\}.
\end{align*}

	It is clear that $E_0$ and $E$ are Banach spaces for the norms $\|\cdot\|_{E_0}$ and $\|\cdot\|_{E}$, where
\begin{align*}
	\|(q,h)\|_{E_0} =&~ \bigl(\|\rho_1\mathcal{M}_2(q)\|^2_{L^2((0,L)\times(0,T))} +\|\rho_2q\|^2_{L^2((0,L)\times(0,T))}
	+\|\rho_3h1_\om\|^2_{L^2((0,L)\times(0,T))}\\
	&+\|\rho_4^{1/2}(\mathcal{M}_2(q),q)\|^2_{H^1(0,T;L^2(0,L))}	
	+ \|\rho_4^{1/2} q\|^2_{ L^2(0,T;H^2_\nu(0,L))\cap L^{\infty}(0,T;H^1(0,L))}   )^{1/2}
\end{align*}
	and
\begin{align*}
 	\|(q,h)\|_{E}= &~\bigl(\|(q,h)\|^2_{E_0} +\|\rho_4[\mathcal{M}_1(\mathcal{M}_2(q)) + Aq - h1_{\omega}]\|^2_{L^2((0,L)\times(0,T))}\\
	& \|\rho_4^{1/2} \mathcal{M}_1(q)\|^2_{L^4((0,L)\times(0,T))}   + \|\rho_4^{1/3} \mathcal{M}_1(q)\|^2_{L^6((0,L)\times(0,T))} +
	\|q(\cdot,0)\|^2_{H^2_\nu(0,L)} \\
	&\|\rho_4^{1/2} \mathcal{M}_1(q)\|^2_{W^{1,2}(0,T;(H^1(0,L))')\cap  L^2(0,T;H^1(0,L))\cap  L^\infty(0,T;L^2(0,L))}
	)^{1/2}.
\end{align*}

	{ We will now present a null controllability result for \eqref{eq:monoNulllinear}. This will be crucial to deduce controllability for the nonlinear system \eqref{eq:monoNull1} in the next section.}
\begin{propo}\label{nullcontrol}
	Let the assumptions of Theorem~\emph{\ref{teo1}} hold, and assume that
	$p_0\in L^{2}(0,L)$, $q_0\in H^{2}_\nu(0,L)$, and
\begin{equation}\label{eq:G}
	\rho_4G\in L^2((0,L)\times(0,T)).
\end{equation}
	Then there exists a control $h$, with $\textnormal{supp}\,h(\cdot,t)\subset\om(t)\,~\forall t\in(0,T)$, such that for the solution $(q,p)$ to \eqref{eq:monoNulllinear}
	we have that $(q,h)\in E$.
	In particular, we have that
\begin{equation}\label{eq:nullcontrol_p_q}
	p(\cdot,T)=q(\cdot,T)=0.
\end{equation}
\end{propo}
\begin{proof}
	We will follow the general method introduced and used in~\cite{fursikov-imanuvilov} for
	linear parabolic problems. The existence proof will be based in a Lax-Milgram argument.
	To motivate the introduction of the appropriate bilinear form it is useful to introduce
	the following auxiliary extremal problem
\begin{equation}\label{optimalcontrol}
\left\{
\begin{array}{llr}
	\dis \hbox{Minimize }   \ J(\theta,q,h) =  \frac{1}{2}\iiTL\! \left[\rho_1^2|\theta|^2+\rho_2^2|q|^2 +\rho_3^2|h|^21_{\om}\right]\,dx\,dt																\\
	\noalign{\smallskip}
	\dis \hbox{subject to }h\in L^2((0,L)\times(0,T)),\text{ supp}(h(\cdot,t))\subset\om(t)\,~\forall t\in(0,T)\text{ and } \\
	\noalign{\smallskip}
\left\{
\begin{array}{lcl}
			\mathcal{M}_1(\theta) + A(x,t)\,q- h1_{\omega}=G 		&  \mbox{in}&	(0,L) \times (0,T),		\\
			\noalign{\smallskip}\dis
			\mathcal{M}_2(q)  =  \theta					&  \mbox{in}&	(0,L) \times (0,T),		\\
			\noalign{\smallskip}\dis
			\sigma_i(x)q_{x}\big|_{x=0}=\sigma_i(x)q_{x}\big|_{x=L} = 0				& \mbox{on}&	 (0,T),	\\
			\noalign{\smallskip}\dis
			\theta(\cdot,0) =\theta_0,\quad q(\cdot,0) =q_0 		& \mbox{in}&	(0,L),
\end{array}
\right.
\end{array}
\right.
\end{equation}
	where $\theta_0=p_0-(\sigma_i (x)q_{0,x})_x +\ell_p(\bar v_0,\bar w_0)q_0\in L^2(0,L)$.

	Observe that due to the behavior of the weights $\rho_i$ at $t=T$
	 a solution $(\widehat{\theta},\widehat{q},\widehat{h})$ to \eqref{optimalcontrol} is a good candidate
	to satisfy  $(\widehat{q},\widehat{h})\in E$.

	Let us suppose for the moment that $(\widehat{\theta},\widehat{q},\widehat{h})$ solves \eqref{optimalcontrol}.
	Then, by the {\it Lagrange's multipliers formalism} the dual variables $\widehat{\varphi}$
	and $\widehat{\psi}$ satisfy the following system
\begin{equation}\label{eq:lagrange_charac}
\left\{
\begin{array}{lcl}
	\widehat{\theta}=\rho_1^{-2}[\mathcal{M}_1^*(\widehat{\varphi})- \widehat{\psi}]	& \text{in}& (0,L) \times (0,T), \\
	\noalign{\smallskip}
	\widehat{q}= \rho_2^{-2}[M^*_2(\widehat{\psi})+ A\widehat{\varphi}]           	& \text{in}& (0,L) \times (0,T), \\
	\noalign{\smallskip}
	\widehat{h}=-\rho_3^{-2}\widehat{\varphi}1_\om         					& \text{in}& (0,L) \times (0,T), \\
	\noalign{\smallskip}
	\sigma_i(x)\widehat{\psi}_{x}\big|_{x=0}=\sigma_i(x)\widehat{\psi}_{x}\big|_{x=L} = 0			& \mbox{on}&	 (0,T),	\\
\end{array}
\right.
   \end{equation}
	where $\mathcal{M}_i^*$ is the adjoint operator of $\mathcal{M}_i$ $(i=1,2)$, i.e.,
$$
	\mathcal{M}_1^*(\varphi)=-\varphi_t , \quad\hbox{and}\quad \mathcal{M}_2^*(\psi)=-\psi_t  +K\psi.
$$
	
	Now let us set:
$$
\begin{array}{l}
	\dis P_0 = \bigg\{\, (\varphi,\psi)\in C^{2}([0,L] \times [0,T];\mathbb{R}^2) :
	\sigma_i\psi_{x}\big|_{x=0}=\sigma_i\psi_{x}\big|_{x=L}= 0 \,~\mbox{in}\,~(0,T)\,\bigg\},
\end{array}
$$
	the bilinear form on $P_0\times P_0$
\begin{eqnarray*}
	\mathcal{B}((\widetilde{\varphi},\widetilde{\psi}),(\varphi,\psi))\hspace{-0.5cm}&& \dis =  \iiTL\rho_1^{-2}[\mathcal{M}^*_1(\widetilde{\varphi})- \widetilde{\psi}]
	[\mathcal{M}^*_1(\varphi)-\psi]\,dx\,dt
	\\
	\noalign{\smallskip}
	& & \quad \dis
	+ \iiTL\rho_2^{-2}[\mathcal{M}^*_2(\widetilde{\psi})+A \widetilde{\varphi}]
	[\mathcal{M}^*_2(\psi)+A \varphi]\,dx\,dt + \int_0^T \!\!\!\int_{\om(t)} \rho_3^{-2}\widehat{\varphi}\,\varphi\,dx\,dt
\end{eqnarray*}
	and the linear form on $P_0$
\begin{eqnarray*}
	\noalign{\smallskip}\dis
	\langle \mathcal{L},(\varphi,\psi)\rangle
	= \iiTL\! G\,\varphi \,dx\,dt
	+\int_0^L\theta_0\,\varphi(0)\,dx+\int_0^L q_0\,\psi(0)\,dx\nonumber.
\end{eqnarray*}
	For these definitions we note that $(\widehat{\varphi},\widehat{\psi})$ should satisfies
\begin{equation}\label{bilinear}
	\mathcal{B}((\widehat{\varphi},\widehat{\psi}),(\varphi,\psi))
	= \langle \mathcal{L},(\varphi,\psi)\rangle \quad \forall(\varphi,\psi)\in P_0,
   \end{equation}
	i.e.,~the solution to~\eqref{optimalcontrol} satisfies \eqref{bilinear}.
	Conversely, if we are able to solve \eqref{bilinear} in  suitable sense and then use \eqref{eq:lagrange_charac} to define
	$(\widehat{\theta},\widehat{q},\widehat{h})$, then we will be able to prove that we have found a solution to~\eqref{optimalcontrol}.

	Next we will focus on the Lax-Milgram problem \eqref{bilinear}.
	It is clear that $\mathcal{B}: P_0\times P_0 \to \mathbb{R}$ is a symmetric, definite positive and bilinear
	form on~$P_0$, i.e.~a scalar product in this linear space (thanks to the Carleman estimate~\eqref{carleman_0}). We will denote by $P$ the completion of $P_0$ for the
	norm induced by $\mathcal{B}$. Then $P$ is a Hilbert space for the scalar product $\mathcal{B}$.
	On the other hand, in view of the Carleman estimate~\eqref{carleman_0}, \eqref{eq:G} and the fact that $(\bar\xi)^{-1/2} e^{s\bar\alpha}\leq\rho_4$, the linear form
	$(\varphi,\psi)\mapsto \langle \mathcal{L},(\varphi,\psi)\rangle$ is well-defined and continuous on $P$.
	Hence, from {\it Lax-Milgram's lemma,} we deduce that the variational problem
\begin{equation}\label{eq:lax_milgram_pb}
			\mathcal{B}((\widehat{\varphi},\widehat{\psi}),(\varphi,\psi))
			=\langle \mathcal{L},(\varphi,\psi)\rangle \quad 	\forall (\varphi,\psi)\in P,
\end{equation}
	possesses exactly one solution $(\widehat{\varphi},\widehat{\psi})\in P$.

	With $(\widehat{\varphi},\widehat{\psi})$ given let $\widehat{\theta}$, $\widehat{q}$ and $\widehat{h}$ be given by \eqref{eq:lagrange_charac}.
   	It is readily seen that
\begin{equation}\label{eq:Jhat}
	J(\widehat{\theta},\widehat{q},\widehat{h})={1\over2}\mathcal{B}((\widehat{\varphi},\widehat{\psi}),(\widehat{\varphi},\widehat{\psi}))<+\infty
\end{equation}
	and, also, that $(\widehat{\theta},\widehat{q})$ is the unique weak solution to the system in
	\eqref{optimalcontrol} for $h =\widehat{h}$.
	
	Finally, it remains to prove that $(\widehat q,\widehat h)\in E$. Using \eqref{eq:Jhat} and
	the linear system in \eqref{optimalcontrol}, we can easily check that $
	\rho_1\mathcal{M}_2(\widehat q)$, $\rho_2 \widehat q$, $\rho_3\widehat h1_\om\in L^2((0,L)\times(0,T))$,
		$\rho_4[\mathcal{M}_1(\mathcal{M}_2(\widehat q)) + A\widehat q -\widehat  h1_{\omega}]\in L^2((0,L)\times(0,T))$ and
		$\widehat q(\cdot,0)\in H^2_\nu(0,L)$. We set $(\tilde \theta,\tilde q,\tilde h):=\rho_4^{1/2}( \widehat \theta, \widehat q, \widehat h)$ and
		we can see that $(\widetilde \theta,\widetilde q,\widetilde h)$ solves the following problem
\begin{equation}\label{eq:tilde q}
\left\{
\begin{array}{lcl}
			\mathcal{M}_1(\widetilde\theta) = - A(x,t)\,\widetilde q+ \rho_4^{1/2} G+ \widetilde h1_{\omega} +(\rho_4^{1/2})_{t}\widehat \theta		&  \mbox{in}&	(0,L) \times (0,T),		\\
			\noalign{\smallskip}\dis
			\mathcal{M}_2(\widetilde q)  =  \widetilde \theta+(\rho_4^{1/2})_{t} \widehat q					&  \mbox{in}&	(0,L) \times (0,T),		\\
			\noalign{\smallskip}\dis
			\sigma_i(x)\widetilde q_{x}\big|_{x=0}=\sigma_i(x)\widetilde q_{x}\big|_{x=L} = 0				& \mbox{on}&	 (0,T),	\\
			\noalign{\smallskip}\dis
			\widetilde\theta(\cdot,0) =\rho_4^{1/2} (0)\theta_0,\quad \widetilde q(\cdot,0) =\rho_4^{1/2}(0) q_0 		& \mbox{in}&	(0,L).
\end{array}
\right.
\end{equation}
		Notice that, thanks to the fact that $\rho_1\mathcal{M}_2(\widehat q)$ and $\rho_2 \widehat q$
		belong to $L^2((0,L)\times(0,T))$,
	we deduce  that $(\rho_4^{1/2})_{t}\widehat q, (\rho_4^{1/2})_{t}\widehat \theta\in L^2((0,L)\times(0,T))$.
	Thus, since $\theta_0\in L^2(0,L)$, $ q_0\in H^2_\nu(0,L)$, we have that
\begin{equation}\label{reg_tilde}
\left\{
	\begin{alignedat}{2}
		&\widetilde\theta\in W^{1,2}(0,T;L^2(0,L))\\
		&\widetilde q \in W^{1,2}(0,T;L^2(0,L))\cap L^2(0,T;H^2_\nu(0,L))\cap L^\infty(0,T;H^1(0,L)).
	\end{alignedat}
	\right.
\end{equation}	
	
	Notice that the linear system in \eqref{optimalcontrol} is equivalent to \eqref{eq:monoNulllinear} for $(\widehat p,\widehat q)$,
	with $\widehat p=\widehat q_t$,
	and control $h =\widehat{h}$. In fact, by a similar argument as used to obtain \eqref{reg_tilde},
	we get that $\rho_4^{3/4}p\in L^2((0,L)\times(0,T))$.
	We set $(\widetilde p,\widetilde q,\widetilde h):=\rho_4^{\bar r}( \widehat p, \widehat q, \widehat h)$ (with $\bar r=1/3$ and $\bar r=1/2$) and
		we observe that $(\widetilde p,\widetilde q,\widetilde h)$ solves the following problem
\begin{equation}\label{eq:monoNulllineartildep}
    \!\!\!\!\!\left\{
        \begin{array}{lcl}
        \widetilde p_t-(\sigma_i(x) \widetilde p_x)_x +\ell_p(\bar v,\bar w)\widetilde p+ \ell_q(\bar v,\bar w) \widetilde q
         = \rho_4^{\bar r} G+\widetilde h1_{\omega} +(\rho_4^{\bar r})_{t}\widehat p        &  \mbox{in}&    (0,L) \times (0,T),        \\
            \noalign{\smallskip}\dis
            \widetilde q_t  =   \widetilde p+(\rho_4^{\bar r})_{t}\widehat q                    &  \mbox{in}& (0,L) \times (0,T),        \\
            \noalign{\smallskip}\dis
            \sigma_i(x)\widetilde p_x\big|_{x=0}=\sigma_i(x)\widetilde p_x\big|_{x=L} = 0                & \mbox{on}& (0,T),    \\
            \noalign{\smallskip}\dis
            \widetilde p(\cdot,0) =\rho_4^{\bar r}(0)p_0,\quad \widetilde q(\cdot,0) =\rho_4^{\bar r} (0)q_0         & \mbox{in}&    (0,L),
        \end{array}
    \right.
\end{equation}
	
		Since $p_0\in L^2(0,L)$, $ q_0\in H^2_\nu(0,L)$ and
		$(\rho_4^{\bar r})_{t} \widehat q, (\rho_4^{\bar r})_{t}\widehat p\in L^2((0,L)\times(0,T))$,
		we deduce that
\begin{equation}\label{reg_tilde1}
\left\{
	\begin{alignedat}{2}
		&\widetilde q\in W^{1,2}(0,T;L^2(0,L))\\
		&\widetilde p \in L^2(0,T;H^1(0,L))\cap W^{1,2}(0,T;(H^1(0,L))')\cap L^\infty(0,T;L^2(0,L)).
	\end{alignedat}
	\right.
\end{equation}	
	We conclude using \cite[Chapter II. $\S3$]{Lady} to get  $\widetilde p\in L^{\tilde\kappa}((0,L)\times (0,T))$, for $1\leq {\tilde\kappa}\leq 6$.
	
	Finally, we need to argue that \eqref{eq:nullcontrol_p_q} holds. For this propose, we use that $(\widehat q,\widehat h)\in E$
	and consequently $(\rho_2\widehat q,\rho_1\widehat \theta)\in L^2((0,L)\times(0,T))\times L^2((0,L)\times(0,T))$.
	Since $\widehat q,\widehat \theta\in C^0([0,T];L^2(0,L))$ the singularities of $\rho_1$ and $\rho_2$
	imply \eqref{eq:nullcontrol_p_q}.
\end{proof}


 \subsection{Controllability for the nonlinear monodomain model}\label{Section:monodomainnonlinear}

	{ In this section we will prove Theorem ~\ref{teo1} using the results obtained
	in the previous section which allow us to locally invert a nonlinear equation.
	For the latter we rely on the {\it Lyusternik-Graves inverse mapping theorem}, see \cite[Chapter $2$, p. $107$]{A-T}:}
\begin{thm} \label{inversemaptheo}
	Let $B_1$ and $B_2$ be two Banach spaces and let $ \mathcal{A} : B_1 \mapsto B_2 $ satisfy $ \mathcal{A} \in C^1(B_1;B_2)$.
	Assume that $e_0 \in B_1$, $\mathcal{A}(e_0) = i_0 $ and $\mathcal{A}'(e_0):B_1 \mapsto B_2 $ is surjective. Then,
	there exists $ \delta > 0$ such that, for every $ i \in B_2$ satisfying $\| i - i_0\|_{B_2} < \delta$, there exists one solution
	to the equation
$$
	\mathcal{A} (e) = i, \quad e \in B_1.
$$
\end{thm}
	
	We shall apply this result with $B_1=E$, $B_2=F_1\times F_2$ and for any $e = (q,h) \in B_1$ we set
$$
	\mathcal{A}(e) =(\mathcal{M}_1(\mathcal{M}_2(q)) + Aq - h1_{\omega}+\mathcal{N}(\mathcal{M}_1(q),q), q(0),\mathcal{M}_1(q)(0)).
$$
	Here, $F_1=\rho_4^{-1} L^2((0,L)\times(0,T))$ and $F_2= H^2_\nu(0,L)\times L^2(0,L)$.

	Thanks to the definition of the space $E$, it is not difficult to check that $\mathcal{A}$, which contains linear, bilinear and trilinear terms, is continuous and therefore $\mathcal{A}\in C^1(B_1;B_2)$. Let $e_0$ be the
	origin of~$B$. Notice that $\mathcal{A}'(e_0): B_1 \mapsto B_2$ is the mapping that, to each $e=(q,h)\in B_1$, associates the
	function $(\mathcal{M}_1(\mathcal{M}_2(q)) + Aq - h1_{\omega} , q(0),\mathcal{M}_1(q)(0))$ in~$B_2$. In view of the null controllability result for \eqref{eq:monoNulllinear}
	given in Proposition \ref{nullcontrol},  $\mathcal{A}'(e_0)$ is surjective.

	Consequently, we can indeed apply Theorem \ref{inversemaptheo} with these data and, in particular, there exists
	$\delta >0$ such that, if
\[
	\|(0,q_0,p_0)\|_{B_2}=\|(q_0,p_0)\|_{H^2_\nu(0,L)\times L^2(0,L)} \leq \delta,
\]
	we can find a control $h$, with $\textnormal{supp}\,h(\cdot,t)\subset\om(t)\,~\forall t\in(0,T)$, such that the associated solution to~\eqref{eq:monoNull1} satisfies $p(\cdot,T) =0$
	and $q(\cdot,T)=0$ in $(0,L)$.
	
	\
	
	This concludes the proof of Theorem~\ref{teo1}.

\

	\begin{rmq}\label{rmq_AA}{ Do to its practical relevance we have  chosen to focus our attention on the PDE-ODE coupled system which arises from the monodomain equations. However, controllability result that we obtained for $\eqref{eq:monoT}$ can be extended to cover more general situations. For this purpose we assume that
\begin{equation}\label{eq:kk1}
F\in C^2(\mathbb{R}^2;\mathbb{R}),\;  B\in W^{1,\infty}(0,L)  \text{ and  }C\in W^{1,\infty}(0,T;L^\infty(0,L))
\end{equation}
and consider
 	\[
    \left\{
        \begin{array}{lcl}
        v_t-(\sigma_i(x) v_x)_x+{ B(x)v_x+C(x,t)v}={ F(v,w)}+\mathcal{I}_{s,i}{-\mathcal{I}_{s,e}} &\mbox{in}&(0,L) \times (0,T),    \\
        \noalign{\smallskip}\dis
        w_t   +g(v,w)= 0 &\mbox{in}&(0,L) \times (0,T),    \\
        \noalign{\smallskip}\dis
        \sigma_i(x)v_x\big|_{x=0}=\sigma_i(x)v_x\big|_{x=L} = 0 &\mbox{on}&(0,T),            \\
        \noalign{\smallskip}\dis
        v(\cdot,0) =v_0,\quad w(\cdot,0) =w_0 &\mbox{in}&(0,L).
        \end{array}
    \right.
\] 
We further fix a {\it trajectory} $(\bar{v},\bar{w})$ which satisfies 
\begin{equation}\label{eq:kk2}
\bar v \text{ and } \bar w \in W^{1,\infty}(0,T;L^\infty(0,L))
\end{equation}

and is a solution to the related uncontrolled system\,
\[
    \left\{
        \begin{array}{lcl}
        \bar v_t
        -(\sigma_i(x) \bar v_x)_x+{ B(x)\bar v_x+C(x,t)\bar v}
        =F(\bar v,\bar w)+\mathcal{I}_{s,i}            &\mbox{in}&(0,L) \times (0,T),    \\
        \noalign{\smallskip}\dis
        \bar w_t   +g(\bar v,\bar w)= 0         &\mbox{in}&(0,L) \times (0,T),    \\
        \noalign{\smallskip}\dis
        \sigma_i(x)\bar v_x\big|_{x=0}=\sigma_i(x)\bar v_x\big|_{x=L} = 0    &\mbox{on}&(0,T),            \\
        \noalign{\smallskip}\dis
        \bar v(\cdot,0) =\bar v_0,\quad \bar w(\cdot,0) =\bar w_0           &\mbox{in}&(0,L).
        \end{array}
    \right.
\]
	Then, we can relate the exact controllability to trajectories for $(v,w)$ to the null controllability
	of a linearized system
\[
    \!\!\!\!\!\left\{
        \begin{array}{lcl}
         \!\!\!p_t-(\sigma_i(x) p_x)_x +{ B(x) p_x+\widetilde C(x,t)p}
         ={\partial F\over \partial v}(\bar v,\bar w)p +\gamma{\partial F\over \partial w}(\bar v,\bar w)q +G+h1_{\omega}         &  \mbox{in}&    (0,L) \times (0,T),      \\
            \noalign{\smallskip}\dis
            q_t  =   p                    &  \mbox{in}& (0,L) \times (0,T),        \\
            \noalign{\smallskip}\dis
            \sigma_i(x)p_x\big|_{x=0}=\sigma_i(x)p_x\big|_{x=L} = 0            & \mbox{on}& (0,T),    \\
            \noalign{\smallskip}\dis
            p(\cdot,0) =p_0,\quad q(\cdot,0) =q_0         & \mbox{in}&   (0,L),
        \end{array}
    \right.
\]
    where $G$ belongs to a space of functions that decay exponentially as $t\rightarrow T^-$
    and $\widetilde C(x,t):=C(x,t)-\gamma\beta$ and ${ h:=-\gamma e^{\gamma\beta t}\mathcal{I}_{s,e}}$
    is the control and $1_\omega$ is the characteristic function of $\omega$.

	Considering the change of variables $\theta =q_t-(\sigma_i(x)q_x)_x+{ B(x)q_x+[\widetilde C(x,t)-{\partial F\over \partial v}(\bar v,\bar w)]q}$,  the null controllability for the linearized system is equivalent to the observability
	of the following adjoint system
\[
	\left\{
		\begin{array}{lcl}
			- \varphi_t  -\psi=R 					&\mbox{in}&	 (0,L) \times (0,T), 		\\
			\noalign{\smallskip}\dis
			-\psi_t-(\sigma_i(x)\psi_x)_x-({ B(x)\psi})_x+{\left[\widetilde C(x,t)-{\partial F\over \partial v}(\bar v,\bar w)\right]}\psi +{\widehat C(x,t)}\varphi=S 	&\mbox{in}&	 (0,L) \times (0,T),  		\\
			\noalign{\smallskip}\dis
			\sigma_i(x)\psi_{x}\big|_{x=0}=\sigma_i(x)\psi_{x}\big|_{x=L} = 0     				&\mbox{on}&	(0,T), 	\\
			\noalign{\smallskip}\dis
			\varphi(T) = \varphi_T,\quad\psi(T) = \psi_T 		&\mbox{in}&	(0,L),
		\end{array}
	\right.
\]
	where
	$$
	\widehat C(x,t)=\partial_t\!\!\left[{\partial F\over \partial v}(\bar v,\bar w)-\widetilde A(x,t)\right]-\gamma{\partial F\over \partial w}(\bar v,\bar w)={\partial^2 F\over \partial v^2}(\bar v,\bar w)\bar v_t+
	{\partial^2 F\over \partial w \partial v}(\bar v,\bar w)\bar w_t-\widetilde A_t(x,t)-\gamma{\partial F\over \partial w}(\bar v,\bar w).$$
	 Now, with \eqref{eq:kk1} and \eqref{eq:kk2} holding one can verify  a Carleman inequality as in Theorem \ref{thm:Carl_mono} and we obtain exact null controllability of the linearized system.
	
	The proof of exact controllability of the nonlinear system to $(\bar v, \bar w)$  depends on the
the specific growth properties of $F$.  
The spaces  $E$, $F_1$ and $F_2$ reflecting the properties of $F$,  need to be construncted,  such that the Nemytskii operator associated to $F$ is
	$C^1(E;F_1)$ and the solution for the linear system, with $G\in F_1$ and initial conditions in $F_2$, belongs to $E$.}
\end{rmq}	

%
%



\begin{appendix}
%
%
%


\section{Neumann Carleman inequality for a parabolic equation}\label{App:mov_lapla}

	To verify Lemma \ref{lemma:elliptic} we start with the following remark.
\begin{rmq}\label{rmk:g}
	Notice that the function $r$ blows up at $t=0$ and $t=T$.
	Also, observe that $\partial^k_tr(t)={(-1)^k k!\over t^{k+1}}=(-1)^k k!(r(t))^{k+1}$ close to $t=0$ and
	$\partial^k_tr(t)={k!\over (T-t)^{k+1}}=k!(r(t))^{k+1}$ close $t=T$.
\end{rmq}

	Further we have
\begin{equation}\label{eq:alpha_deriv}
\!\!\!\!\!\!
\begin{alignedat}{2}
	\alpha_x =&~\lambda\xi(-\eta_x),~ \\
	\alpha_{xx}=&~\lambda^2\xi(-\eta_x^2 -\lambda^{-1}\eta_{xx}),\\
	\alpha_{xxx}=&~\lambda^3\xi(- \eta_x^3 - 3\lambda^{-1} \eta_{xx}\eta_x-\lambda^{-2}\eta_{xxx}),\\
	\alpha_{xxxx}=&~\lambda^4\xi(- \eta_x^4 - 6\lambda^{-1}\eta_x^2\eta_{xx}- 3\lambda^{-2}\eta_{xx}^2
	-\lambda^{-3}\eta_{xxxx}),\\
	\alpha_t =&~\lambda\xi^2\left[-\xi^{-1}\eta_t+\lambda^{-1}\left(e^{2\lambda(\|\eta\|_\infty-\eta)}-e^{-\lambda\eta}\right)\left({1\over r} \right)_t\right],\\
	\alpha_{tt}=&~\lambda^2\xi^3\!\!\left[-\xi^{-2}\eta^2_t-\lambda^{-1}\xi^{-2}\eta_{tt}
	-2\lambda^{-1}e^{-2\lambda\eta}\eta_t{r_{t}\over r^3}+\lambda^{-2}\left(e^{\lambda(2\|\eta\|_\infty-3\eta)}-e^{-2\lambda\eta}\right)
	{r_{tt}\over r^3}\right],\\
	\alpha_{xt} =&~\lambda^2\xi^2\left[-\xi^{-1}\eta_t\eta_x-\lambda^{-1}\xi^{-1}\eta_{xt}+\lambda^{-1}e^{-\lambda\eta}(r^{-1})_t \eta_x\right].
	\end{alignedat}
\end{equation}
	It follows that there exists a positive constant $C>0$, such that for $(x,t)\in [0,L]\times(0,T)$
	we have the following pointwise estimates:
\begin{equation}\label{eq:alpha_esti}
\begin{alignedat}{2}
	|\alpha_{xx}|\leq&~C\lambda^2\xi,\\
	|\alpha_{xxx}|\leq&~C\lambda^3\xi,\\
	|\alpha_{xxxx}|\leq&~C\lambda^4\xi,\\
	|\alpha_t| \leq&~C(T+e^{2\lambda\|\eta\|_\infty})\lambda\xi^2,\\
	|\alpha_{tt}| \leq&~C(T^2+T+e^{2\lambda\|\eta\|_\infty})\lambda^2\xi^3,\\
	|\alpha_{xt}| \leq&~C(T+1)\lambda^2\xi^2.
	\end{alignedat}
\end{equation}
	
	We set $w=e^{-s\alpha}\psi$ and observe that
$$
	w_x=-s\alpha_xw+e^{-s\alpha}\psi_x.
$$
	Using the boundary conditions of $\psi$, we deduce $w_x=-s \alpha_x w $ on $\{0,L\}\times(0,T)$.
\begin{rmq}\label{rmk:initialcondition}
	From the definition of $\alpha$, given in \eqref{alpha}, notice that
	$w(\cdot,T)=w(\cdot,0)=0$ and $w_x(\cdot,T)=w_x(\cdot,0)=0$.
\end{rmq}		
	
	Now, let us introduce the partial differential operator
	$P:=\partial_t+\partial_{x}(\sigma_i\partial_x)$. Then, we have the following decomposition
\[
	e^{-s\alpha} P(e^{s\alpha} w ) = P_e\, w + P_k\,w,
\]
	where
\[
	P_{e} w:=\partial_{x}(\sigma_i\partial_xw)+(s\alpha_t+s^2\sigma_i\alpha_x^2)w
\]
	is the self-adjoint part of the operator $P$ and
\[
	P_{k}w:=w_t+2s\sigma_i\alpha_xw_x+s\partial_{x}(\sigma_i\partial_x\alpha) w
\]
	is the skew-adjoint part of $P$. Without loss of generality we can suppose $\sigma_i=1$. Then we have
$$
	P=\partial_t+\partial_{xx},~
	P_{e} w=w_{xx}+(s\alpha_t+s^2\alpha_x^2)w~
	\hbox{and}~
	P_{k}w=w_t+2s\alpha_xw_x+s\alpha_{xx} w.
$$
It follows that
\begin{equation}\label{eq:L2norm}
\begin{alignedat}{2}
	\|e^{-s\alpha} P(e^{s\alpha} w )\|^2_{L^2((0,L)\times(0,T))}=&~\|P_e w\|^2_{L^2((0,L)\times(0,T))}+\|P_k w\|^2_{L^2((0,L)\times(0,T))}\\
										&+2(P_e w,P_k w)_{L^2((0,L)\times(0,T))}.
\end{alignedat}
\end{equation}

	The rest of the proof is devoted to analyzing the term $(P_e w,P_k w)_{L^2((0,L)\times(0,T))}$. Indeed, from the above definition
	of the operators $P_e$ and $P_k$ it follows
\begin{equation}\label{eq:pepk}
\begin{alignedat}{2}
	2(P_e w, P_k w)_{L^2((0,L)\times(0,T))} =	&~2\left(w_{xx}, w_t\right)+2\left(w_{xx},2s\alpha_x w_x\right)
						+2\left(w_{xx},s\alpha_{xx} w\right)\\
						&~+2\left(s\alpha_tw+s^2\alpha_x ^2w,w_t\right)
						+2\left(s\alpha_tw+s^2 \alpha _x^2w,2s\alpha_x w_x\right)\\
						&~+2\left(s\alpha_tw+s^2\alpha_x^2w,s\alpha_{xx} w\right)   \\
	=:&~I_1 + I_2 + I_3 + I_4+I_5+I_6.
\end{alignedat}
\end{equation}

	Now, in order to get estimates for the term $2(P_e w, P_k w)_{L^2((0,L)\times(0,T))}$,
	first let us work with each integral term $I_i$, $i=1,\ldots,6$.
	
	For the first integral term, we integrate by parts in time and we obtain
\[
	I_1=- 2\ii w_{x}w_{xt}+2\int_0^T
	      w_t  w_x\bigg|_{x=0}^{x=L}.
\]
	Then, thanks to Remark \ref{rmk:initialcondition} and the fact that $w_x=-s \alpha_x w $ on $\{0,L\}\times(0,T)$,
	after an integration by parts in the last term we have
\begin{equation}\label{eq:I_1}
\begin{alignedat}{2}
	I_1=&~s\int_0^T \alpha_{xt}w^2\bigg|_{x=0}^{x=L}.
\end{alignedat}
\end{equation}

	For the second term, we integrate by parts in space and we deduce
\begin{equation}\label{eq:I_2}
\begin{alignedat}{2}
	I_2
	=&~-2s\ii \alpha_{xx}w_x^2 +2s\int_0^T \alpha_xw_x^2\bigg|_{x=0}^{x=L}.
\end{alignedat}
\end{equation}

	For the third term,  after two integration by parts, we obtain
\[
	I_3=- 2s \ii \alpha_{xxx}w_xw  -2s \ii \alpha_{xx}  w_x^2+2s\int_0^T \alpha_{xx} w  w_x\bigg|_{x=0}^{x=L}.
\]

	And again integrating by parts in space the first term, we deduce
\begin{equation}\label{eq:I_3}
\begin{alignedat}{2}
	I_3=&~ s \ii \alpha_{xxxx}w^2  -2s \ii \alpha_{xx}  w_x^2
	 +2s\int_0^T \alpha_{xx} w  w_x\bigg|_{x=0}^{x=L} -s\int_0^T \alpha_{xxx}w^2\bigg|_{x=0}^{x=L}.
\end{alignedat}
\end{equation}
	
	Fourthly, we integrate by parts in time and using Remark  \ref{rmk:initialcondition}, we get
\begin{equation}\label{eq:I_4}
\begin{alignedat}{2}
	I_4=&~  -s\ii \alpha_{tt} w^2-2s^2\ii \alpha_x\alpha_{xt} w^2.
\end{alignedat}
\end{equation}		
	
		And for the fifth term, we conclude
\begin{equation}\label{eq:I_5}
\begin{alignedat}{2}
	I_5=&~ -\ii \left(2s^2\alpha_x\alpha_{xt}+4s^3\alpha_x^2\alpha_{xx}\right)w^2-
	      2\ii (s^2\alpha_t+s^3\alpha_x^2)\alpha_{xx} w^2\\
	      &~+2\int_0^T (s^2\alpha_t\alpha_x+s^3\alpha_x^3)  w^2\bigg|_{x=0}^{x=L}  .
\end{alignedat}
\end{equation}

	For the last term, we obtain	
\begin{equation}\label{eq:I_6}
\begin{alignedat}{2}
	I_6=&~ 2\ii (s^2\alpha_t+s^3\alpha_x^2) \alpha_{xx} w^2.
\end{alignedat}
\end{equation}

From \eqref{eq:pepk}-\eqref{eq:I_6}, we  get
\begin{align*}
	2(P_ew,P_kw) =&~-4s\ii \alpha_{xx}w_x^2+ \ii \left(s\alpha_{xxxx}-s\alpha_{tt}-4s^2\alpha_x\alpha_{xt}
	-4s^3\alpha_x^2\alpha_{xx} \right)w^2\\
	&+\int_0^T\left(2s\alpha_x w_x^2 +2s \alpha_{xx} ww_x+2s^2\alpha_t\alpha_xw^2+2s^3\alpha_x^3w^2+s\alpha_{xt}w^2
	- s\alpha_{xxx}w^2 \right)\bigg|_{x=0}^{x=L}.
\end{align*}
	
	Recalling that  $w_x=-s \alpha_x w $ on $\{0,L\}\times(0,T)$, from the previous identity we deduce
\begin{equation}\label{eq:inner}
\begin{alignedat}{2}	
	2(P_ew,P_kw)=&~-4s\ii \alpha_{xx}w_x^2+ \ii \left(s\alpha_{xxxx}-s\alpha_{tt}-4s^2\alpha_x\alpha_{xt}
	-4s^3\alpha_x^2\alpha_{xx} \right)w^2\\
	&+\int_0^T\left[4s^3\alpha_x^3 +2s^2\alpha_x\left(\alpha_t - \alpha_{xx}\right) +s(\alpha_t-\alpha_{xx})_x\right]
	w^2\bigg|_{x=0}^{x=L} \\
	&= DT_1+DT_2+BT,
\end{alignedat}
\end{equation}
	where $DT_1$ and $DT_2$ correspond to the distributed terms and $BT$ corresponds to the boundary terms.

Then, thanks to the identities \eqref{eq:alpha_deriv}, we obtain
\begin{align*}
	BT=&~-4s^3\lambda^3\int_0^T\eta_x^3\xi^3 w^2\bigg|_{x=0}^{x=L}
	-2s^2\lambda\int_0^T\left(\alpha_t - \alpha_{xx}\right)\eta_x\xi w^2\bigg|_{x=0}^{x=L}
	+s\int_0^T (\alpha_t-\alpha_{xx})_xw^2\bigg|_{x=0}^{x=L}.
\end{align*}	
	Therefore, using the estimates \eqref{eq:alpha_esti} and the property \eqref{P7}-\eqref{P8} of $\eta$,
	we have the following bound
\begin{align*}
	BT\geq&~4C^3s^3\lambda^3\int_0^T\xi^3 w^2\big|_{x=L}
	-2s^2\lambda^2C^2\left(T+e^{2\lambda\|\eta\|_\infty}\right)\int_0^T\xi^3 w^2\big|_{x=L}
	-2s^2\lambda^3C^2\int_0^T\xi^2 w^2\big|_{x=L}\\
	&~-C(T+1)s\lambda^2\int_0^T \xi^2w^2\big|_{x=L}
	-Cs\lambda^3\int_0^T \xi w^2\bigg|_{x=L}\\
	&~+4C^3s^3\lambda^3\int_0^T\xi^3 w^2\big|_{x=0}
	-2s^2\lambda^2C^2\left(T+e^{2\lambda\|\eta\|_\infty}\right)\int_0^T\xi^3 w^2\big|_{x=0}
	-2s^2\lambda^3C^2\int_0^T\xi^2 w^2\big|_{x=0}\\
	&~-C(T+1)s\lambda^2\int_0^T \xi^2w^2\big|_{x=0}
	-Cs\lambda^3\int_0^T \xi w^2\bigg|_{x=0}.
\end{align*}

	Hence, we have for any $\lambda\geq C$ and any
	$s\geq C\left(1+T+e^{2\lambda\|\eta\|_\infty}\right)$:
\begin{equation}\label{eq:BT}
	BT\geq4C^3s^3\lambda^3\int_0^T\xi^3 w^2\big|_{x=L}+4C^3s^3\lambda^3\int_0^T\xi^3 w^2\big|_{x=0}.
\end{equation}

	Now, let us estimate the distributed terms in $DT$. Thanks to \eqref{P1} and \eqref{eq:alpha_deriv},
	for $DT_1$ we have:
\begin{align*}\label{eq:BT}
	DT_1=&~4s\lambda^2\ii \eta_x^2\xi w_x^2
	+4s\lambda\ii \eta_{xx}\xi w_x^2\\
	\geq&~Cs\lambda^2\ii \xi w_x^2 - Cs\lambda^2\int_0^T\!\!\!\int_{\om_0(t)}\xi w_x^2
	-Cs\lambda\ii \xi w_x^2.
\end{align*}
	Hence, taking $\lambda\geq C$, we obtain
\begin{equation}\label{eq:DT1}
	DT_1\geq Cs\lambda^2\ii \xi w_x^2 - Cs\lambda^2\int_0^T\!\!\!\int_{\om_0(t)}\xi w_x^2.
\end{equation}

	Also, in order to get an estimate for $DT_2$, we use
	\eqref{P1}, \eqref{P2}, \eqref{eq:alpha_deriv} and \eqref{eq:alpha_esti} to obtain:
\begin{align*}\label{eq:BT}
	DT_2
	\geq&~ Cs^3\lambda^4\ii\xi^3 w^2- Cs^3\lambda^4\int_0^T\!\!\!\int_{\om_0(t)}\xi^3 w^2
	-Cs^3\lambda^3\ii\xi^3 w^2\\
	&~-Cs\lambda^4\ii \xi w^2-Cs\lambda^2\ii \xi^3 w^2 - Cs^2\lambda^3\ii\xi^3 w^2.
\end{align*}	
	Therefore, for $\lambda\geq C$, we deduce
\begin{equation}\label{eq:DT2}
	DT_2
	\geq Cs^3\lambda^4\ii\xi^3 w^2- Cs^3\lambda^4\int_0^T\!\!\!\int_{\om_0(t)}\xi^3 w^2.
\end{equation}	

	From \eqref{eq:L2norm}, \eqref{eq:inner}, \eqref{eq:BT},  \eqref{eq:DT1} and \eqref{eq:DT2}, we
	conclude that
\begin{equation}\label{eq:concl1}
\begin{alignedat}{2}
	&~\|P_e w\|^2_{L^2((0,L)\times(0,T))}+\|P_k w\|^2_{L^2((0,L)\times(0,T))}+s^3\lambda^4\ii\xi^3 w^2+s\lambda^2\ii \xi w_x^2\\
	&~+s^3\lambda^3\int_0^T\xi^3 w^2\big|_{x=L}+s^3\lambda^3\int_0^T\xi^3 w^2\big|_{x=0}\\
	\leq&~ C\left(\|e^{-s\alpha} P(e^{s\alpha} w )\|^2_{L^2((0,L)\times(0,T))} +
	s\lambda^2\int_0^T\!\!\!\int_{\om_0(t)}\xi w_x^2
	+s^3\lambda^4\int_0^T\!\!\!\int_{\om_0(t)}\xi^3 w^2\right).
\end{alignedat}
\end{equation}
	
	Now, using that $P_{e} w=w_{xx}+(s\alpha_t+s^2\alpha_x^2)w$, we can deduce:
\begin{equation}\label{eq:laplac}
\begin{alignedat}{2}
	s^{-1}\ii\xi^{-1} w^2_{xx}
	=&~ s^{-1}\ii\xi^{-1}  |P_e w-(s\alpha_t+s^2\alpha_x^2)w|^2\\
	\leq&~ C s^{-1}\ii\xi^{-1}\left(  |P_e w|^2+ s^2\alpha_t^2w^2+s^4\alpha_x^4w^2\right)\\
	\leq&~ C s^{-1}\ii\xi^{-1}\left(  |P_e w|^2+ s^2\lambda^2\xi^4w^2+s^4\lambda^4\xi^4w^2\right)\\
	\leq&~ C\left( s^{-1}\ii\xi^{-1} |P_e w|^2+\ii s^3\lambda^4\xi^3w^2\right).
\end{alignedat}
\end{equation}

	We can do the same for $P_{k}w:=w_t+2s\alpha_xw_x+s\alpha_{xx} w$ and then
\begin{equation}\label{eq:time}
\begin{alignedat}{2}
	s^{-1}\ii\xi^{-1} w^2_{t}
	=&~ s^{-1}\ii\xi^{-1}  |P_k w-2s\alpha_xw_x-s\alpha_{xx} w|^2\\
	\leq&~ C s^{-1}\ii\xi^{-1}\left(  |P_k w|^2+ s^2\alpha_x^2w^2_x+s^2\alpha_{xx}^2w^2\right)\\
	\leq&~ C s^{-1}\ii\xi^{-1}\left(  |P_k w|^2+ s^2\lambda^2\xi^2w^2_x+s^4\lambda^4\xi^2w^2\right)\\
	\leq&~ C\left( s^{-1}\ii\xi^{-1} |P_k w|^2+\ii s\lambda^2\xi^2w^2_x+\ii s^3\lambda^4\xi^3w^2\right).
\end{alignedat}
\end{equation}

	From \eqref{eq:concl1}, \eqref{eq:laplac} and \eqref{eq:time}, we obtain
\begin{equation}\label{eq:concl4}
\begin{alignedat}{2}
	&~s^{-1}\ii\xi^{-1} w^2_{t}+s^{-1}\ii\xi^{-1} w^2_{xx}+s^3\lambda^4\ii\xi^3 w^2+s\lambda^2\ii \xi w_x^2\\
	&~+s^3\lambda^3\int_0^T\xi^3 w^2\big|_{x=L}+s^3\lambda^3\int_0^T\xi^3 w^2\big|_{x=0}\\
	\leq&~ C\left(\|e^{-s\alpha} P(e^{s\alpha} w )\|^2_{L^2((0,L)\times(0,T))} +
	s\lambda^2\int_0^T\!\!\!\int_{\om_0(t)}\xi w_x^2
	+s^3\lambda^4\int_0^T\!\!\!\int_{\om_0(t)}\xi^3 w^2\right).
\end{alignedat}
\end{equation}

	To conclude the proof, we need to eliminate the local term $w_x$ containing on the right hand side of the previous inequality.
	To do that, let
	us introduce a
	function $\bar\zeta \in C^\infty([0,L]\times[0,T])$ satisfying
\[
	\begin{array}{lll}
	&0\le \bar\zeta\le 1&\forall (x,t)\in [0,L]\times[0,T],\\
	\noalign{\smallskip}\dis
	&\bar\zeta(x,t) =1&\forall x\in\om_0(t),~\forall t\in[0,T],\\
	\noalign{\smallskip}\dis
	&\bar\zeta(x,t) =0& \forall x\in [0,L]\setminus \overline{\om_1(t)},~ \forall t\in[0,T],
	\end{array}
\]
	where we use \eqref{eq:om1om0}.
	
	We have
\[
\begin{alignedat}{2}
	s\lambda^2\int_0^T\!\!\!\int_{\om_0(t)}\xi w_x^2
	\leq&~ s\lambda^2\ii \bar\zeta\xi w_x^2 \\
	=&~ - s\lambda^2\ii \bar\zeta\xi w_{xx}w- s\lambda^2\ii \bar\zeta_x\xi w_{x}w- s\lambda^2\ii \bar\zeta\xi_x w_{x}w\\
	&~+s\lambda^2\int_0^T\xi w_xw\bigg|^{x=L}_{x=0}.
\end{alignedat}
\]	
	Using the fact that $w_x=-s \alpha_x w $ on $\{0,1\}\times(0,T)$ and \eqref{P7}-\eqref{P8} and \eqref{eq:alpha_deriv}, we deduce
\[
\begin{alignedat}{2}
	Cs\lambda^2\int_0^T\!\!\!\int_{\om_0(t)}\xi w_x^2
	\leq&~ - Cs\lambda^2\ii \bar\zeta\xi w_{xx}w- Cs\lambda^2\ii \bar\zeta_x\xi w_{x}w- Cs\lambda^3\ii \bar\zeta\eta_x w_{x}w\\
	&~+Cs^2\lambda^3\int_0^T\xi^2 \eta_xw^2\bigg|^{x=L}_{x=0}\\
	\leq&~ {1\over2} s^{-1}\ii\xi^{-1} w_{xx}^2+{1\over2} s\lambda^2\ii\xi w_{x}^2
	+{C\over2} s^3\lambda^4\ii (\bar\zeta^2+\bar\zeta_x^2) \xi^3w^2.
\end{alignedat}
\]	
	Then, we obtain from the above inequality and \eqref{eq:concl1}
\begin{equation}\label{eq:concl2}
\begin{alignedat}{2}
	&~s^{-1}\ii\xi^{-1} w^2_{t}+s^{-1}\ii\xi^{-1} w^2_{xx}+s^3\lambda^4\ii\xi^3 w^2+s\lambda^2\ii \xi w_x^2\\
	&~+s^3\lambda^3\int_0^T\xi^3 w^2\big|_{x=L}+s^3\lambda^3\int_0^T\xi^3 w^2\big|_{x=0}\\
	\leq&~ C\left(\|e^{-s\alpha} P(e^{s\alpha} w )\|^2_{L^2((0,L)\times(0,T))} +s^3\lambda^4\int_0^T\!\!\!\int_{\om_1(t)}\xi^3 w^2\right).
\end{alignedat}
\end{equation}

	We finally can turn back to $\psi$:
\begin{equation}\label{eq:concl2}
\begin{alignedat}{2}
	&~s^{-1}\ii\!\!\!\!\xi^{-1} \psi^2_{t}e^{-2s\alpha}+s^{-1}\ii\!\!\!\!\xi^{-1} \psi^2_{xx}e^{-2s\alpha}+s\lambda^2\ii \!\!\!\!\xi \psi_x^2e^{-2s\alpha}+s^3\lambda^4\ii\!\!\!\!\xi^3 \psi^2e^{-2s\alpha}\\
	&~+s^3\lambda^3\int_0^T\xi^3 \psi^2e^{-2s\alpha}\big|_{x=L}+s^3\lambda^3\int_0^T\xi^3 \psi^2e^{-2s\alpha}\big|_{x=0}\\
	&~\leq C\left(\ii e^{-2s\alpha} f^2+s^3\lambda^4\int_0^T\!\!\!\int_{\om_1(t)}\xi^3 \psi^2e^{-2s\alpha}\right)
\end{alignedat}
\end{equation}
	and hence \eqref{carleman:neumann} follows.
\end{appendix}

\section*{Acknowledgements}
This work has been partially done while the second author was visiting the
		Karl-Franzens-Universit\"at (Graz, Austria).
		He wishes to thank the members of the Institut f\"ur Mathematik und Wissenschaftliches Rechnen for their kind hospitality.


\end{document}